\documentclass{article}

\usepackage{times}
\usepackage{amssymb,amsmath,amsthm}
\usepackage{mathrsfs}
\usepackage{epsfig}
\usepackage[usenames,dvipsnames]{color}

\newcommand\ie{{\em i.e.}}

\def\B{\mathscr B}
\def\C{\mathbb C}
\def\d{\mathrm{d}}

\def\H{\mathcal H}

\def\L{\mathscr L}

\def\N{\mathbb N}

\def\Q{\mathbb Q}
\def\R{\mathbb R}

\def\T{\mathbb T}
\def\U{\mathscr U}

\def\Z{\mathbb Z}

\def\dom{\mathcal D}

\def\div{\mathop{\mathrm{div}}\nolimits}
\def\e{\mathop{\mathrm{e}}\nolimits}

\def\det{\mathop{\mathrm{det}}\nolimits}

\def\re{\mathop{\mathsf{Re}}\nolimits}
\def\linf{\mathsf{L}^{\:\!\!\infty}}
\def\ltwo{\mathsf{L}^{\:\!\!2}}

\def\slim{\mathop{\hbox{\rm s-}\lim}\nolimits}

\newtheorem{Theorem}{Theorem}[section]
\newtheorem{Remark}[Theorem]{Remark}
\newtheorem{Lemma}[Theorem]{Lemma}
\newtheorem{Assumption}[Theorem]{Assumption}
\newtheorem{Corollary}[Theorem]{Corollary}

\begin{document}

%--------------------------------------------------------------------------------------
% Title
%--------------------------------------------------------------------------------------

\title{Commutator methods for the spectral analysis of uniquely ergodic dynamical
systems}

\author{R. Tiedra de Aldecoa\footnote{Supported by the Chilean Fondecyt Grant 1090008
and by the Iniciativa Cientifica Milenio ICM P07-027-F ``Mathematical Theory of
Quantum and Classical Magnetic Systems'' from the Chilean Ministry of Economy.}
}
\date{\small}
\maketitle \vspace{-1cm}

\begin{quote}
\emph{
\begin{itemize}
\item[] Facultad de Matem\'aticas, Pontificia Universidad Cat\'olica de Chile,\\
Av. Vicu\~na Mackenna 4860, Santiago, Chile
\item[] \emph{E-mail\hspace{1pt}:} rtiedra@mat.puc.cl
\end{itemize}
}
\end{quote}

%--------------------------------------------------------------------------------------
% Abstract
%--------------------------------------------------------------------------------------

\begin{abstract}
We present a method, based on commutator methods, for the spectral analysis of
uniquely ergodic dynamical systems. When applicable, it leads to the absolute
continuity of the spectrum of the corresponding unitary operators. As an illustration,
we consider time changes of horocycle flows, skew products over translations and
Furstenberg transformations. For time changes of horocycle flows, we obtain absolute
continuity under assumptions weaker than the ones to be found in the literature.
\end{abstract}

\textbf{2010 Mathematics Subject Classification\hspace{1pt}:} 37A30, 37C10, 37C40,
37D40, 58J51, 81Q10.

\smallskip

\textbf{Keywords\hspace{1pt}:} Commutator methods, continuous spectrum, uniquely
ergodic, horocycle flow, skew products, Furstenberg transformations.

%--------------------------------------------------------------------------------------
\section{Introduction}
\setcounter{equation}{0}
%--------------------------------------------------------------------------------------

Commutator methods in the sense of \'E. Mourre \cite{Mou81} and their various
extensions (see for instance \cite{ABG,BM97,DG97,GGM04,JMP84,PSS81,Sah97_2}) are a
very efficient tool for the spectral analysis of self-adjoint operators. They have
been fruitfully applied to numerous models in mathematical physics and pure
mathematics. Even so, it is only recently that an analogue of the theory has been
developed for unitary operators (see \cite{ABCF06} for the first article on the
topic, and \cite{FRT12} for an optimal formulation of the theory). Accordingly, the
absence of general works on commutator methods for the spectral analysis of unitary
operators from ergodic theory is not a surprise. Our purpose here is to start filling
this gap by presenting an abstract class of uniquely ergodic dynamical systems to
which commutator methods apply. The class in question is simple enough to be
described in terms of commutators and general enough to cover interesting examples of
uniquely ergodic dynamical systems. We hope that the examples treated in this paper,
together with the simplicity of our approach, will motivate other works on commutator
methods for the spectral analysis of (uniquely ergodic) dynamical systems. 

The content of the paper is the following. In Section \ref{Sec_com}, we recall the
needed definitions and results on commutator methods, both for self-adjoint and
unitary operators. Then, we exhibit a general class of unitary operators which are
shown to have purely absolutely continuous spectrum thanks to commutator methods.
Also, we explain why typical examples of such unitary operators are Koopman operators
induced by uniquely ergodic transformations. After that, we dedicate the rest of the
paper to applications of the theory of Section \ref{Sec_com}. We consider time
changes of horocycle flows in Section \ref{sec_horo}, skew products over translations
in Section \ref{sec_skew} and Furstenberg transformations in Section \ref{Sec_Fur}.

In Theorem \ref{thm_spec} of Section \ref{sec_horo}, we show that time changes of
horocycle flows on compact surfaces of constant negative curvature have purely
absolutely continuous spectrum in the orthocomplement of the constant functions. Our
result holds for time changes of class $C^3$, which is the weakest regularity
assumption under which this absolute continuity has been established (see the
discussion after Theorem \ref{thm_spec} for a comparison with recent results of G.
Forni and C. Ulcigrai \cite{FU12} and of the author \cite{Tie12}). In Theorem
\ref{spec_co} of Section \ref{sec_skew}, we prove that skew products over
translations on compact metric abelian Banach Lie groups have countable Lebesgue
spectrum in the orthocomplement of functions depending only on the first variable.
Our result holds for cocycle functions being differentiable along the flow generated
by the translations and with corresponding derivative being Dini-continuous (see
Assumption \ref{ass_phi} for details). In the case of skew products on tori, this
complements previous results of A. Iwanik, M. Lema\'nzyk and D. Rudolph
\cite{{Iwa97_1},ILR93} in one dimension and a previous result of A. Iwanik
\cite{Iwa97_2} in higher dimensions (see Corollary \ref{cor_tori} and the discussion
that follows). Finally, in Theorem \ref{Thm_Fur} of Section \ref{Sec_Fur}, we show
that Furstenberg transformations have countable Lebesgue spectrum in the
orthocomplement of functions depending only on the first variable for perturbations
satisfying a Dini condition. This result is not optimal (in \cite[Cor.~3]{ILR93} the
perturbations are only of bounded variation), but our proof is of independent
interest since it is completely new and does not rely on the study of the Fourier
coefficients of the spectral measure.

%--------------------------------------------------------------------------------------
\section{Commutator methods for uniquely ergodic dynamical systems}\label{Sec_com}
\setcounter{equation}{0}
%--------------------------------------------------------------------------------------

We present in this section a method, based on commutator methods, for the spectral
analysis of uniquely ergodic dynamical systems. We start by recalling some facts on
commutator methods borrowed from \cite{ABG}, \cite{FRT12} and \cite{Sah97_2} (see
also the original paper \cite{Mou81} of \'E. Mourre).

Let $\H$ be a Hilbert space with scalar product
$\langle\hspace{1pt}\cdot\hspace{1pt},\hspace{1pt}\cdot\hspace{1pt}\rangle$ linear in
the second argument, denote by $\B(\H)$ the set of bounded linear operators on $\H$,
and write $\|\hspace{1pt}\cdot\hspace{1pt}\|$ for the norm on $\H$ and the norm on
$\B(\H)$. Let also $A$ be a self-adjoint operator in $\H$ with domain $\dom(A)$, and
take $S\in\B(\H)$. For any $k\in\N$, we say that $S$ belongs to $C^k(A)$, with
notation $S\in C^k(A)$, if the map
\begin{equation}\label{eq_group}
\R\ni t\mapsto\e^{-itA}S\e^{itA}\in\B(\H)
\end{equation}
is strongly of class $C^k$. In the case $k=1$, one has $S\in C^1(A)$ if and only if
the quadratic form
$$
\dom(A)\ni\varphi\mapsto\big\langle\varphi,iSA\varphi\big\rangle
-\big\langle A\varphi,iS\varphi\big\rangle\in\C
$$
is continuous for the topology induced by $\H$ on $\dom(A)$. We denote by
$[\hspace{1pt}iS,A]$ the bounded operator associated with the continuous extension of
this form, or equivalently the strong derivative of the function \eqref{eq_group} at
$t=0$.

A condition slightly stronger than the inclusion $S\in C^1(A)$ is provided by the
following definition\hspace{1pt}: $S$ belongs to $C^{1+0}(A)$, with notation
$S\in C^{1+0}(A)$, if $S\in C^1(A)$ and
$$
\int_0^1\frac{\d t}t\hspace{1pt}\big\|\e^{-itA}[A,S]\e^{itA}-[A,S]\big\|<\infty.
$$
If we regard $C^1(A)$, $C^{1+0}(A)$ and $C^2(A)$ as subspaces of $\B(\H)$, then we
have the inclusions
$$
C^2(A)\subset C^{1+0}(A)\subset C^1(A)\subset\B(\H).
$$

Now, if $H$ is a self-adjoint operator in $\H$ with domain $\dom(H)$ and spectrum
$\sigma(H)$, we say that $H$ is of class $C^k(A)$ if $(H-z)^{-1}\in C^k(A)$ for some
$z\in\C\setminus\sigma(H)$. So, $H$ is of class $C^1(A)$ if and only if the quadratic
form
$$
\dom(A)\ni\varphi\mapsto
\big\langle\varphi,(H-z)^{-1}A\hspace{1pt}\varphi\big\rangle
-\big\langle A\hspace{1pt}\varphi,(H-z)^{-1}\varphi\big\rangle\in\C
$$
extends continuously to a bounded form defined by the operator
$[(H-z)^{-1},A]\in\B(\H)$. In such a case, the set $\dom(H)\cap\dom(A)$ is a core for
$H$ and the quadratic form
$$
\dom(H)\cap\dom(A)\ni\varphi\mapsto\big\langle H\varphi,A\varphi\big\rangle
-\big\langle A\varphi,H\varphi\big\rangle\in\C
$$
is continuous in the topology of $\dom(H)$ \cite[Thm.~6.2.10(b)]{ABG}. This form then
extends uniquely to a continuous quadratic form on $\dom(H)$ which can be identified
with a continuous operator $[H,A]$ from $\dom(H)$ to the adjoint space $\dom(H)^*$.
In addition, the following relation holds in $\B(\H)$\hspace{1pt}:
\begin{equation}\label{2com}
\big[(H-z)^{-1},A\big]=-(H-z)^{-1}[H,A](H-z)^{-1}.
\end{equation}

Let $E^H(\hspace{1pt}\cdot\hspace{1pt})$ denote the spectral measure of the
self-adjoint operator $H$, and assume that $H$ is of class $C^1(A)$. Then, for each
bounded Borel set $J\subset\R$ the operator $E^H(J)[\hspace{1pt}iH,A]E^H(J)$ is
bounded and self-adjoint. If there exist a number $a>0$ and a compact operator
$K\in\B(\H)$ such that
\begin{equation}\label{Mourre_H}
E^H(J)[\hspace{1pt}iH,A]E^H(J)\ge aE^H(J)+K,
\end{equation}
then one says that $H$ satisfies a Mourre estimate on $J$ and that $A$ is a conjugate
operator for $H$ on $J$. Also, one says that $H$ satisfies a strict Mourre estimate
on $J$ if \eqref{Mourre_H} holds with $K=0$. The main consequence of a strict Mourre
estimate is to imply a limiting absorption principle for $H$ on $J$ if $H$ is also of
class $C^{1+0}(A)$. This in turns implies that $H$ has no singular spectrum in $J$.
If $H$ only satisfies a Mourre estimate on $J$, then the same holds up to the
possible presence of a finite number of eigenvalues in $J$, each one of finite
multiplicity. We recall here a version of these results (see \cite[Sec.~7.1.2]{ABG}
and \cite[Thm.~0.1]{Sah97_2} for more details)\hspace{1pt}:

\begin{Theorem}\label{thm_abs}
Let $H$ and $A$ be self-ajoint operators in a Hilbert space $\H$, with $H$ of class
$C^{1+0}(A)$. Suppose there exist a bounded Borel set $J\subset\R$, a number $a>0$
and a compact operator $K\in\B(\H)$ such that
\begin{equation}\label{Mourre_H_bis}
E^H(J)[\hspace{1pt}iH,A]E^H(J)\ge aE^H(J)+K.
\end{equation}
Then, $H$ has at most finitely many eigenvalues in $J$, each one of finite
multiplicity, and $H$ has no singularly continuous spectrum in $J$. Furthermore, if
\eqref{Mourre_H_bis} holds with $K=0$, then $H$ has no singular spectrum in $J$.
\end{Theorem}

Similar notations and results exist in the case of a unitary operator $U\in C^1(A)$
with spectral measure $E^U(\hspace{1pt}\cdot\hspace{1pt})$ and spectrum
$\sigma(U)\subset\mathbb S^1\equiv\{z\in\C\mid|z|=1\}$. Namely, one says that $U$
satisfies a Mourre estimate on a Borel set $\Theta\subset\mathbb S^1$ if there exists
a number $a>0$ and a compact operator $K\in\B(\H)$ such that
\begin{equation}\label{Mourre_U}
E^U(\Theta)\hspace{1pt}U^*[A,U]E^U(\Theta)\ge aE^U(\Theta)+K.
\end{equation}
Also, one says that $U$ satisfies a strict Mourre estimate on $\Theta$ if
\eqref{Mourre_U} holds with $K=0$. Furthermore, one has the following result on the
spectral nature of $U$ (see \cite[Thm.~2.7 \& Rem.~2.8]{FRT12} for a more general
version of this result)\hspace{1pt}:

\begin{Theorem}\label{thm_unitary}
Let $U$ and $A$ be respectively a unitary and a self-ajoint operator in a Hilbert
space $\H$, with $U\in C^{1+0}(A)$. Suppose there exist an open set
$\Theta\subset\mathbb S^1$, a number $a>0$ and a compact operator $K\in\B(\H)$ such
that
\begin{equation}\label{Mourre_U_bis}
E^U(\Theta)\hspace{1pt}U^*[A,U]E^U(\Theta)\ge aE^U(\Theta)+K.
\end{equation}
Then, $U$ has at most finitely many eigenvalues in $\Theta$, each one of finite
multiplicity, and $U$ has no singularly continuous spectrum in $\Theta$. Furthermore,
if \eqref{Mourre_U_bis} holds with $K=0$, then $U$ has no singular spectrum in
$\Theta$.
\end{Theorem}

\begin{Remark}
If $\hspace{1pt}U=\e^{-iH}$ for some bounded self-adjoint operator $H\in C^{1+0}(A)$,
then one can use indifferently the self-adjoint or the unitary formulation of
commutator methods. Indeed, in such a case one has for each $\varphi\in\H$ that
$$
\big(U^*\e^{itA}U\e^{-itA}-1\big)\varphi
=\e^{iH}\big[\e^{itA},\e^{-iH}\big]\e^{-itA}\varphi
=i\int_0^1\d s\,\e^{isH}\int_0^t\d u\,\e^{iuA}[\hspace{1pt}iH,A]
\e^{i(t-u)A}\e^{-isH}\varphi
$$
which implies that
$$
\left\|\frac{U^*\e^{itA}U\e^{-itA}-1}t
-i\int_0^1\d s\,\e^{isH}[\hspace{1pt}iH,A]\e^{-isH}\right\|
\le\sup_{u\in[0,t]}
\big\|\e^{iuA}[\hspace{1pt}iH,A]\e^{i(t-u)A}-[\hspace{1pt}iH,A]\big\|.
$$
So, one infers that $U\in C^1(A)$ with
$U^*[A,U]=\int_0^1\d s\,\e^{isH}[\hspace{1pt}iH,A]\e^{-isH}$, which in turns implies
the inclusion $U\in C^{1+0}(A)$. Moreover, one has for any Borel set $J\subset\R$ the
equality
$$
E^H(J)=E^U(\Theta)
\quad\hbox{with}\quad
\Theta:=\big\{\e^{-i\lambda}\in\mathbb S^1\mid\lambda\in J\big\}.
$$
Therefore, one obtains the following equivalences of Mourre estimates:
\begin{align*}
&E^H(J)[\hspace{1pt}iH,A]E^H(J)\ge aE^H(J)+K
\quad\hbox{with $J\subset\R$ a bounded Borel set}\\
&\iff E^U(\Theta)\hspace{1pt}U^*[A,U]E^U(\Theta)\ge aE^U(\Theta)
+\int_0^1\d s\,\e^{isH}K\e^{-isH}
\quad\hbox{with $\Theta:=\big\{\e^{-i\lambda}\in\mathbb S^1\mid\lambda\in J\big\}$.}
\end{align*}
\end{Remark}

Now, suppose for a moment that there exists a self-adjoint operator $A$ with domain
$\dom(A)$ such that $U\in C^1(A)$ and $[A,U]=U$. Then, one has $U\in C^k(A)$ for each
$k\in\N$ and $U^*[A,U]=1$. In particular, $U\in C^{1+0}(A)$ and $U$ satisfies a
strict Mourre estimate on all of $\mathbb S^1$. Thus, Theorem \ref{thm_unitary}
applies and one deduces that $U$ has purely absolutely continuous spectrum. In fact,
one can check that the conditions $U\in C^1(A)$ and $[A,U]=U$ imply that
$\e^{-itA}U\e^{itA}=\e^{-it}U$ for each $t\in\R$. So, the operator $U$ is unitarily
equivalent to $\e^{-it}U$ for each $t\in\R$, and thus has purely Lebesgue spectrum
covering the whole circle $\mathbb S^1$. No need of commutator methods whatsoever.

But, now assume that the situation is more general in the sense that we are only able
to find a self-adjoint operator $A$ such that $U\in C^1(A)$ and
$[A,U]=UF+G\hspace{1pt}U$ for some self-adjoint operators $F,G\in\B(\H)$. In this
case, no simple trick permits to obtain Lebesgue spectrum (since it would be
obviously wrong). Moreover, we only get the relation
$$
U^*[A,U]=F+U^*G\hspace{1pt}U
$$
which do not lead to any explicit Mourre estimate, unless we impose some positivity
condition on the operator $F+U^*G\hspace{1pt}U$. Fortunately, in certain situations,
it is sufficient to modify appropriately the operator $A$ in order to get the desired
positivity. Let us explain how to proceed. Since $U\in C^1(A)$, we know from standard
results (see \cite[Prop.~5.1.5-5.1.6]{ABG}) that $U^k\in C^1(A)$ and
$U^k\hspace{1pt}\dom(A)=\dom(A)$ for each $k\in\Z$. Therefore, the operator
$\frac1n\sum_{k=0}^{n-1}U^{-k}\big[A,U^k\big]$ is bounded for each $n\in\N^*$, and
the operator
$$
A_n\hspace{1pt}\varphi
:=\frac1n\sum_{k=0}^{n-1}U^{-k}A\hspace{1pt}U^k\hspace{1pt}\varphi
=\frac1n\sum_{k=0}^{n-1}U^{-k}\big[A,U^k\big]\varphi
+A\hspace{1pt}\varphi,\quad\varphi\in\dom(A_n):=\dom(A),
$$
is self-adjoint. Furthermore, a simple calculation shows that $U\in C^1(A_n)$ with
\begin{equation}\label{average_Mourre}
U^*[A_n,U]
=U^*\hspace{1pt}\frac1n\sum_{k=0}^{n-1}U^{-k}\hspace{1pt}[A,U]\hspace{1pt}U^k
=\frac1n\sum_{k=0}^{n-1}U^{-k}F\hspace{1pt}U^k
+U^*\left(\frac1n\sum_{k=0}^{n-1}U^{-k}\hspace{1pt}G\hspace{1pt}U^k\right)U
\equiv F_n+U^*\hspace{1pt}G_n\hspace{1pt}U,
\end{equation}
which in turns implies that $U\in C^{1+0}(A_n)$ if the operators $F_n$ and $G_n$
satisfy
\begin{equation}\label{cond_G_n}
\int_0^1\frac{\d t}t\hspace{1pt}\big\|\e^{-itA_n}F_n\e^{itA_n}-F_n\big\|<\infty
\qquad\hbox{and}\qquad
\int_0^1\frac{\d t}t\hspace{1pt}\big\|\e^{-itA_n}G_n\e^{itA_n}-G_n\big\|<\infty.
\end{equation}
Now, even if $F+U^*G\hspace{1pt}U$ is not a strictly positive operator, the averaged
operator $F_n+U^*\hspace{1pt}G_n\hspace{1pt}U$ may converge in norm as $n\to\infty$
to a strictly positive operator. In such a case, the r.h.s. of \eqref{average_Mourre}
would be strictly positive for $n$ big enough. Accordingly, one would obtain a strict
Mourre estimate on all of $\mathbb S^1$, and thus conclude by Theorem
\ref{thm_unitary} that $U$ has purely absolutely continuous spectrum if $F_n$ and
$G_n$ satisfy \eqref{cond_G_n} (if the operator $A_n$ is bounded, one can skip the
verification of \eqref{cond_G_n} thanks to a theorem of C. R. Putnam, see
\cite[Thm.~2.3.2]{Put67}).

The convergence in norm of the averaged operator
$F_n+U^*\hspace{1pt}G_n\hspace{1pt}U$ is similar to the uniform convergence of
Birkhoff sums for uniquely ergodic transformations (it is also similar to the norm
convergence of Birkhoff sums for uniquely ergodic automorphisms of $C^*$-algebras, as
defined in \cite[Sec.~1]{AD09}). Therefore, it is quite natural to particularise the
previous construction to the case where $F$ and $G$ are multiplication operators and
$U$ is a unitary operator generated by a uniquely ergodic transformation. So, let
$T:X\to X$ be a uniquely ergodic homeomorphism on a compact metric space $X$ with
normalised Haar measure $\mu$, and let $U_T$ be the unitary Koopman operator in
$\H:=\ltwo(X,\mu)$ given by
$$
U_T:\H\to\H,\quad\varphi\mapsto\varphi\circ T.
$$
Furthermore, assume that there exists a self-adjoint operator $A$ such that
$U_T\in C^1(A)$ and $[A,U_T]=U_Tf+g\hspace{1pt}U_T$ for some functions
$f,g\in C(X;\R)$ (here we identify the functions $f$ and $g$ with the corresponding
multiplication operators). Then, \eqref{average_Mourre} reduces to
\begin{equation}\label{average_ergodic}
(U_T)^*[A_n,U_T]
=\frac1n\sum_{k=0}^{n-1}f\circ T^{-k}
+(U_T)^*\left(\frac1n\sum_{k=0}^{n-1}g\circ T^{-k}\right)U_T
\equiv f_n+(U_T)^*g_n\hspace{1pt}U_T.
\end{equation}
Since $T$ is uniquely ergodic, the Birkhoff sums $f_n$ and $g_n$ converge uniformly
to $\int_X\d\mu\,f$ and $\int_X\d\mu\,g$, respectively. So, if
$\int_X\d\mu\,(f+g)>0$, then the r.h.s. of \eqref{average_ergodic} is strictly
positive for $n$ big enough. Accordingly, one obtains a strict Mourre estimate on all
of $\mathbb S^1$, and one concludes by Theorem \ref{thm_unitary} that $U_T$ has
purely absolutely continuous spectrum if $f_n$ and $g_n$ satisfy
$$
\int_0^1\frac{\d t}t\hspace{1pt}\big\|\e^{-itA_n}f_n\e^{itA_n}-F_n\big\|<\infty
\qquad\hbox{and}\qquad
\int_0^1\frac{\d t}t\hspace{1pt}\big\|\e^{-itA_n}g_n\e^{itA_n}-G_n\big\|<\infty.
$$

Obviously, this last construction can be adapted to various other situations as when
one allows a compact perturbation, or when the r.h.s. of the identity
$[A,U_T]=U_Tf+g\hspace{1pt}U_T$ involves another combination of operators, or when
one has a continuous flow of homeomorphisms $\{T_t\}_{t\in\R}$ generating a strongly
continuous group of unitary operators $\{U_t\}_{t\in\R}$. In the latter case, it
might be more convenient to replace the discrete averages
$A_n=\frac1n\sum_{k=0}^{n-1}(U_T)^{-k}A\hspace{1pt}(U_T)^k$ by the continuous
averages $A_L:=\frac1L\int_0^L\d t\,U_{-t}A\hspace{1pt}U_{t}$ and to study the
generator $H$ of the group $\{U_{t}\}_{t\in\R}$ (instead of the group itself) using
the usual self-adjoint formulation of commutators methods.

%--------------------------------------------------------------------------------------
\section{Time changes of horocycle flows}\label{sec_horo}
\setcounter{equation}{0}
%--------------------------------------------------------------------------------------

Let $\Sigma$ be a compact Riemann surface of genus $\ge2$ and let $M:=T^1\Sigma$ be
the unit tangent bundle of $\Sigma$. The compact $3$-manifold $M$ carries a
probability measure $\mu_\Omega$ (induced by a canonical volume form $\Omega$) which
is preserved by  two distinguished one-parameter groups of
diffeomorphisms\hspace{1pt}: the horocycle flow $\{F_{1,t}\}_{t\in\R}$ and the
geodesic flow $\{F_{2,t}\}_{t\in\R}$. Both flows correspond to right translations on
$M$ when $M$ is identified with a homogeneous space $\Gamma\setminus{\sf PSL}(2;\R)$,
for some cocompact lattice $\Gamma$ in ${\sf PSL}(2;\R)$ (see \cite[Sec.~II.3 \&
Sec.~IV.1]{BM00}). We write $U_j(t)$ ($j=1,2$, $t\in\R$) for the operators given by
$$
U_j(t)\hspace{1pt}\varphi:=\varphi\circ F_{j,t},\quad\varphi\in C(M).
$$
One can check that the families $\{U_j(t)\}_{t\in\R}$ define strongly continuous
unitary groups in the Hilbert space $\H:=\ltwo(M,\mu_\Omega)$, and that  
$U_j(t)\hspace{1pt}C^\infty(M)\subset C^\infty(M)$ for each $t\in\R$. It follows from
Nelson's theorem \cite[Prop.~5.3]{Amr09} that the generator of the group
$\{U_j(t)\}_{t\in\R}$
$$
H_j\hspace{1pt}\varphi:=\slim_{t\to0}it^{-1}\big(U_j(t)-1\big)\varphi,
\quad\varphi\in\dom(H_j):=\left\{\varphi\in\H
\mid\lim_{t\to0}|t|^{-1}\big\|\big(U_j(t)-1\big)\varphi\big\|<\infty\right\},
$$
is essentially self-adjoint on $C^\infty(M)$, and one has
$$
H_j\hspace{1pt}\varphi:=-i\hspace{1pt}\L_{X_j}\varphi,\quad\varphi\in C^\infty(M),
$$
with $X_j$ the divergence-free vector field associated to $\{F_{j,t}\}_{t\in\R}$ and
$\L_{X_j}$ the corresponding Lie derivative.

It is a classical result that the horocycle flow $\{F_{1,t}\}_{t\in\R}$ is uniquely
ergodic \cite{Fur73} and mixing of all orders \cite{Mar78}, and that $U_1(t)$ has
countable Lebesgue spectrum for each $t\ne0$ (see \cite[Prop.~2.2]{KT06} and
\cite{Par53}). Moreover, the groups $\{U_1(t)\}_{t\in\R}$ and $\{U_2(t)\}_{t\in\R}$
satisfy the commutation relation
\begin{equation}\label{rel_com}
U_2(s)\hspace{1pt}U_1(t)\hspace{1pt}U_2(-s)=U_1(\e^st),\quad s,t\in\R,
\end{equation}
(here we consider the negative horocycle flow
$\{F_{1,t}\}_{t\in\R}\equiv\{F_{1,t}^-\}_{t\in\R}$, but everything we say can be
adapted to the positive horocycle flow by inverting a sign, see
\cite[Rem.~IV.1.2]{BM00}). By applying the strong derivative $i\hspace{1pt}\d/\d t$
at $t=0$ in \eqref{rel_com}, one gets that $U_2(s)H_1U_2(-s)\varphi=\e^sH_1\varphi$
for each $\varphi\in C^\infty(M)$. Since $C^\infty(M)$ is a core for $H_1$, one
infers that $H_1$ is $H_2$-homogeneous in the sense of \cite{BG91}; namely,
\begin{equation}\label{eq_homo1}
U_2(s)H_1U_2(-s)=\e^sH_1\quad\hbox{on}\quad\dom(H_1).
\end{equation}
It follows that $H_1$ is of class $C^\infty(H_2)$ with
\begin{equation}\label{com_H1H2}
\big[iH_1,H_2\big]=H_1.
\end{equation}

Now, consider a $C^1$ vector field with the same orientation and proportional to
$X_1$, that is, a vector field $fX_1$ with $f\in C^1\big(M;(0,\infty)\big)$. The
vector field $fX_1$ has the same integral curves as $X_1$, but with reparametrised
time coordinate. Indeed, it is known (see \cite[Sec.~1]{Hum74}) that the formula
$$
t=\int_0^{h(p,t)}\frac{\d s}{f\big(F_{1,s}(p)\big)}\hspace{1pt},\quad t\in\R,~p\in M,
$$
defines for each $p\in M$ a strictly increasing function $\R\ni t\mapsto h(p,t)\in\R$
satisfying $h(p,0)=0$ and $\lim_{t\to\pm\infty}h(p,t)=\pm\infty$. Furthermore, the
implicit function theorem implies that the map $t\mapsto h(p,t)$ is $C^1$ with
$\frac\d{\d t}h(p,t)=f\big(F_{1,h(p,t)}(p)\big)$. Therefore, the function
$\R\ni t\mapsto\widetilde F_{1,t}(p)\in M$ given by
$\widetilde F_{1,t}(p):=F_{1,h(p,t)}(p)$ satisfies the initial value problem 
$$
\frac\d{\d t}\hspace{1pt}\widetilde F_1(p,t)
=(fX_1)_{\widetilde F_1(p,t)}\hspace{1pt},
\quad\widetilde F_1(p,0)=p,
$$
meaning that $\{\widetilde F_{1,t}\}_{t\in\R}$ is the flow of $fX_1$ (note that
$\widetilde F_{1,t}(p)$ is of class $C^1$ in the $p$-variable and of class $C^2$ in
the $t$-variable as predicted by the general theory \cite[Sec.~2.1]{AM78}). Since the
divergence $\div_{\Omega/f}(fX_1)$ of $fX_1$ with respect to the volume form
$\Omega/f$ is zero, the operators
$$
\widetilde U_1(t)\hspace{1pt}\varphi:=\varphi\circ\widetilde F_{1,t}\hspace{1pt},
\quad\varphi\in C(M),
$$
define a strongly continuous unitary group $\{\widetilde U_1(t)\}_{t\in\R}$ in the
Hilbert space $\widetilde\H:=\ltwo(M,\mu_\Omega/f)$. The generator
$\widetilde H:=-i\hspace{1pt}\L_{fX_1}$ of $\{\widetilde U_1(t)\}_{t\in\R}$ is
essentially self-adjoint on $C^1(M)\subset\widetilde\H$ due to Nelson's theorem.

In the following lemma, we introduce two auxiliary operators which will be useful for
the spectral analysis of $\widetilde H$.

\begin{Lemma}\label{Lemma_a}
Let $f\in C^1\big(M;(0,\infty)\big)$, then
\begin{enumerate}
\item[(a)] the operator
$$
\U:\H\to\widetilde\H,\quad\varphi\mapsto f^{1/2}\varphi,
$$
is unitary, with adjoint $\U^*:\widetilde\H\to\H$ given by $\U^*\psi=f^{-1/2}\psi$,
\item[(b)] the symmetric operator
$$
H\varphi:=f^{1/2}H_1f^{1/2}\varphi,\quad\varphi\in C^1(M),
$$
is essentially self-adjoint in $\H$, and the closure of $H$ (which we denote by the
same symbol) is unitarily equivalent to $\widetilde H$,
\item[(c)] for each $z\in\C\setminus\R$, the operator $H_1+zf^{-1}$ is invertible
with bounded inverse, and satisfies
\begin{equation}\label{eq_res}
(H+z)^{-1}=f^{-1/2}\big(H_1+zf^{-1}\big)^{-1}f^{-1/2}.
\end{equation}
\end{enumerate}
\end{Lemma}

\begin{proof}
Point (a) follows from a direct calculation taking into account the boundedness of
$f$ from below and from above. For (b), observe that
$$
H\varphi
=f^{-1/2}fH_1f^{1/2}\varphi
=\U^*\widetilde H\U\varphi
$$
for each $\varphi\in\U^*C^1(M)$. So, $H$ is essentially self-adjoint on
$\U^*C^1(M)\equiv C^1(M)$, and the closure of $H$ is unitarily equivalent to
$\widetilde H$. To prove (c), take $z\equiv\lambda+i\mu\in\C\setminus\R$,
$\varphi\in\dom\big(H_1+zf^{-1}\big)\equiv\dom(H_1)$ and
$\{\varphi_n\}\subset C^\infty(M)$ such that
$\lim_n\|\varphi-\varphi_n\|_{\dom(H_1)}=0$. Then, it follows from (b) that
$$
\big\|\big(H_1+zf^{-1}\big)\varphi\big\|^2
=\lim_n\big\|f^{-1/2}(H+z)f^{-1/2}\varphi_n\big\|^2
\ge\inf_{p\in M}f^{-2}(p)\hspace{1pt}\mu^2\big\|\varphi\big\|^2,
$$
and thus $H_1+zf^{-1}$ is invertible with bounded inverse (see
\cite[Lemma~3.1]{Amr09}). Now, to show \eqref{eq_res}, take $\psi=(H+z)\zeta$ with
$\zeta\in C^1(M)$, observe that
\begin{equation}\label{eq_res_bis}
(H+z)^{-1}\psi-f^{-1/2}\big(H_1+zf^{-1}\big)^{-1}f^{-1/2}\psi=0,
\end{equation}
and then use the density of $(H+z)C^1(M)$ in $\H$ to extend the identity
\eqref{eq_res_bis} to all of $\H$.
\end{proof}

The operators $H$ and $\widetilde H$ are unitarily equivalent due to Lemma
\ref{Lemma_a}(b). Therefore, one can either work with $H$ in $\H$ or with
$\widetilde H$ in $\widetilde\H$ to determine the spectral properties associated with
the time change $fX_1$. For convenience, we present our results for the operator $H$.
We start by proving some regularity properties of $f$ and $H$ with respect to $H_2$.
The function
$$
g:=\frac12-\frac12\hspace{1pt}\L_{X_2}\big(\ln(f)\big)
$$
pops up naturally\hspace{1pt}:

\begin{Lemma}\label{lem_regu}
Let $f\in C^1\big(M;(0,\infty)\big)$, $\alpha\in\R$ and $z\in\C\setminus\R$. Then,
\begin{enumerate}
\item[(a)] the multiplication operator $f^\alpha$ satisfies $f^\alpha\in C^1(H_2)$
with $\big[if^\alpha,H_2\big]=-\alpha\hspace{1pt}f^\alpha\L_{X_2}\big(\ln(f)\big)$,

\item[(b)] $(H+z)^{-1}\in C^1(H_2)$ with
$\big[i(H+z)^{-1},H_2\big]=-(H+z)^{-1}(Hg+gH)(H+z)^{-1}$.
\end{enumerate}
\end{Lemma}

\begin{proof}
(a) The chain rule for Lie derivatives and the strict positivity of $f$ imply that
$$
\L_{X_2}(f^\alpha)
=\alpha\hspace{1pt}f^{\alpha-1}\L_{X_2}(f)
=\alpha\hspace{1pt}f^\alpha\L_{X_2}\big(\ln(f)\big).
$$
Thus, one has for each $\varphi\in C^\infty(M)$
$$
\big\langle\varphi,if^\alpha H_2\hspace{1pt}\varphi\big\rangle
-\big\langle H_2\hspace{1pt}\varphi,if^\alpha\varphi\big\rangle
=\big\langle\varphi,\big[if^\alpha,H_2\big]\varphi\big\rangle
=\big\langle\varphi,-\alpha\hspace{1pt}f^\alpha\L_{X_2}\big(\ln(f)\big)
\varphi\big\rangle\hspace{1pt}.
$$
Since $f^\alpha\L_{X_2}\big(\ln(f)\big)\in\linf(M)$, it follows by the density of
$C^\infty(M)$ in $\dom(H_2)$ that $f^\alpha\in C^1(H_2)$ with
$\big[if^\alpha,H_2\big]=-\alpha\hspace{1pt}f^\alpha\L_{X_2}\big(\ln(f)\big)$.

(b) Let $t\in\R$ and $\varphi\in\H$. Then, one infers from Equations \eqref{eq_homo1}
and \eqref{eq_res} that
$$
\e^{-itH_2}(H+z)^{-1}\e^{itH_2}\varphi
=\e^{-itH_2}f^{-1/2}\e^{itH_2}\big(\e^tH_1+z\e^{-itH_2}f^{-1}\e^{itH_2}\big)^{-1}
\e^{-itH_2}f^{-1/2}\e^{itH_2}\varphi.
$$
So, one gets from point (a), Equation \eqref{eq_res} and Lemma \ref{Lemma_a}(b) that
\begin{align*}
&\frac\d{\d t}\hspace{1pt}\e^{-itH_2}(H+z)^{-1}\e^{itH_2}\varphi\Big|_{t=0}\\
&=\big[if^{-1/2},H_2\big]\big(H_1+zf^{-1}\big)^{-1}f^{-1/2}\varphi
+f^{-1/2}\big(H_1+zf^{-1}\big)^{-1}\big[if^{-1/2},H_2\big]\varphi\\
&\qquad-f^{-1/2}\big(H_1+zf^{-1}\big)^{-1}\big(H_1+z\big[if^{-1},H_2\big]\big)
\big(H_1+zf^{-1}\big)^{-1}f^{-1/2}\varphi\\
&=\frac12\hspace{1pt}\L_{X_2}\big(\ln(f)\big)(H+z)^{-1}\varphi
+\frac12\hspace{1pt}(H+z)^{-1}\L_{X_2}\big(\ln(f)\big)\varphi\\
&\qquad-(H+z)^{-1}\big\{H+z\hspace{1pt}\L_{X_2}\big(\ln(f)\big)\big\}(H+z)^{-1}\varphi\\
&=\frac12\hspace{1pt}(H+z)^{-1}H\hspace{1pt}\L_{X_2}\big(\ln(f)\big)(H+z)^{-1}\varphi
+\frac12\hspace{1pt}(H+z)^{-1}\L_{X_2}\big(\ln(f)\big)H(H+z)^{-1}\varphi\\
&\qquad-(H+z)^{-1}H(H+z)^{-1}\varphi\\
&=-(H+z)^{-1}(Hg+gH)(H+z)^{-1}\varphi,
\end{align*}
which implies the claim.
\end{proof}

In \cite{Tie12} we used the operator $H_2$ as a conjugate operator for $H$ (in fact
for $H^2$). This led us to impose, as A. G. Kushnirenko in \cite[Thm.~2]{Kus74}, the
strict positivity of the function $g$ in order to get at some point a strict Mourre
estimate. Here, we will show that this can be avoided if one uses a conjugate
operator taking into account the unique ergodicity of the horocycle flow
$\{F_{1,t}\}_{t\in\R}$\hspace{1pt}, as presented in Section \ref{Sec_com}. We start
with the definition of the new conjugate operator. We use for $L>0$ the notations
$g_L$ and $\widetilde{g_L}$ for the following averages of $g$ along the time-changed
flow $\{\widetilde F_{1,t}\}_{t\in\R}$\hspace{1pt}:
$$
g_L:=\frac1L\int_0^L\d t\,\big(g\circ\widetilde F_{1,-t}\big)
\qquad\hbox{and}\qquad
\widetilde{g_L}:=\frac1L\int_0^L\d t\int_0^t\d s\,
\big(g\circ\widetilde F_{1,-s}\big).
$$

\begin{Lemma}[Conjugate operator]\label{Lem_conj}
Let $f\in C^1\big(M;(0,\infty)\big)$ and $L>0$.
\begin{enumerate}
\item[(a)] For each $\varphi\in C^1(M)$, one has the equality
$$
\frac1L\int_0^L\d t\,\e^{itH}H_2\e^{-itH}\varphi
=\textstyle-i\big(\L_X+\frac12\div_\Omega X\big)\varphi,
$$
with $X:=X_2+2\hspace{1pt}\widetilde{g_L}fX_1$ and
$\div_\Omega X=2\hspace{1pt}\widetilde{g_L}\L_{X_1}(f)+2(g-g_L)$ the divergence of
$X$ relative to the volume form $\Omega$.
\item[(b)] If $f\in C^3\big(M;(0,\infty)\big)$, then the operator
$$
A_L\varphi:=\frac1L\int_0^L\d t\,\e^{itH}H_2\e^{-itH}\varphi,
\quad\varphi\in C^1(M),
$$
is essentially self-adjoint in $\H$ (and its closure is denoted by the same symbol).
\end{enumerate}
\end{Lemma}

\begin{proof}
(a) We start by collecting some information on the function $\widetilde{g_L}$. For
each $s\in\R$, we have
$$
\e^{isH}g\e^{-isH}
=\e^{is\hspace{1pt}\U^*\widetilde H\U}g\e^{-is\hspace{1pt}\U^*\widetilde H\U}
=\U^*\e^{is\widetilde H}g\e^{-is\widetilde H}\U
=g\circ\widetilde F_{1,-s}\hspace{1pt}.
$$
Thus,
$$
\frac1L\int_0^L\d t\int_0^t\d s\,\e^{isH}g\e^{-isH}
=\frac1L\int_0^L\d t\int_0^t\d s\,\big(g\circ\widetilde F_{1,-s}\big)
=\widetilde{g_L}
$$
and
\begin{equation}\label{eq_g_L}
\L_{fX_1}(\widetilde{g_L})
=\frac\d{\d\tau}\frac1L\int_0^L\d t\int_0^t\d s\,
\big(g\circ\widetilde F_{1,-s}\circ\widetilde F_{1,\tau}\big)\bigg|_{\tau=0}
=\frac1L\int_0^L\d t\,\frac\d{\d\tau}\int_{\tau-t}^\tau\d u\,
\big(g\circ\widetilde F_{1,u}\big)\bigg|_{\tau=0}
=g-g_L\hspace{1pt}.
\end{equation}
This implies that $H\widetilde{g_L}\varphi\in\H$ for each $\varphi\in C^1(M)$ since
\begin{equation}\label{dans_dom}
H\widetilde{g_L}\varphi
=\widetilde{g_L}H\varphi+\big[H,\widetilde{g_L}\big]\varphi
=\widetilde{g_L}H\varphi-i\L_{fX_1}(\widetilde{g_L})\varphi
=\widetilde{g_L}H\varphi-i(g-g_L)\hspace{1pt}\varphi.
\end{equation}

Now, take $\varphi\in C^1(M)$ and $\psi\in\dom(H)$, and set
$H_2(\tau):=(i\tau)^{-1}\big(\e^{i\tau H_2}-1\big)$ for each $\tau\in\R$. Then, one
has the equalities
\begin{align*}
&\bigg\langle\psi,\bigg(\frac1L\int_0^L\d t\,\e^{itH}H_2(\tau)\e^{-itH}
-H_2(\tau)\bigg)\varphi\bigg\rangle\\
&=\bigg\langle(H-i)\hspace{1pt}\psi,\frac1L\int_0^L\d t\int_0^t\d s\,
\frac\d{\d s}\hspace{1pt}\e^{isH}(H+i)^{-1}H_2(\tau)(H-i)^{-1}\e^{-isH}(H-i)
\hspace{1pt}\varphi\bigg\rangle\\
&=\bigg\langle(H-i)\hspace{1pt}\psi,\frac1L\int_0^L\d t\int_0^t\d s\,\e^{isH}
(H+i)^{-1}\big[iH,H_2(\tau)\big](H-i)^{-1}\e^{-isH}(H-i)
\hspace{1pt}\varphi\bigg\rangle\\
&=\bigg\langle(H-i)\hspace{1pt}\psi,-\frac1L\int_0^L\d t\int_0^t\d s\,
\e^{isH}(H-i)(H+i)^{-1}\big[i(H-i)^{-1},H_2(\tau)\big]\e^{-isH}(H-i)
\hspace{1pt}\varphi\bigg\rangle.
\end{align*}
But, we know from Lemma \ref{lem_regu}(b) that 
$\slim_{\tau\searrow0}\big[i(H-i)^{-1},H_2(\tau)\big]=-(H-i)^{-1}(Hg+gH)(H-i)^{-1}$
and we know from \eqref{dans_dom} that $H\widetilde{g_L}\hspace{1pt}\varphi\in\H$.
So, one obtains that
\begin{align*}
&\bigg\langle\psi,\frac1L\int_0^L\d t\,\e^{itH}H_2\e^{-itH}\varphi
-H_2\hspace{1pt}\varphi\bigg\rangle\\
&=\bigg\langle(H-i)\hspace{1pt}\psi,\frac1L\int_0^L\d t\int_0^t\d s\,
\e^{isH}(H+i)^{-1}(Hg+gH)(H-i)^{-1}\e^{-isH}(H-i)\hspace{1pt}\varphi\bigg\rangle\\
&=\big\langle(H-i)\hspace{1pt}\psi,(H+i)^{-1}\big(H\widetilde{g_L}
+\widetilde{g_L}H\big)(H-i)^{-1}(H-i)\hspace{1pt}\varphi\big\rangle\\
&=\big\langle\psi,\big(H\widetilde{g_L}
+\widetilde{g_L}H\big)\varphi\big\rangle\hspace{1pt},
\end{align*}
which implies the equality
$$
\frac1L\int_0^L\d t\,\e^{itH}H_2\e^{-itH}\varphi
=H_2\hspace{1pt}\varphi+\big(H\widetilde{g_L}+\widetilde{g_L}H\big)\varphi,
$$
due to the density of $\dom(H)$ in $\H$. Now, the equations
$\div_\Omega(X_1)=\div_\Omega(X_2)=0$ and \eqref{eq_g_L} imply that
$$
\div_\Omega X
=\div_\Omega X_2+\div_\Omega(2\hspace{1pt}\widetilde{g_L}fX_1)
=2\hspace{1pt}\widetilde{g_L}\L_{X_1}(f)+2f\L_{X_1}(\widetilde{g_L})
=2\hspace{1pt}\widetilde{g_L}\L_{X_1}(f)+2(g-g_L).
$$
So, one infers that
\begin{align*}
\frac1L\int_0^L\d t\,\e^{itH}H_2\e^{-itH}\varphi
&=\big(-i\hspace{1pt}\L_{X_2}+2\hspace{1pt}\widetilde{g_L}H
+\big[H,\widetilde{g_L}\big]\big)\varphi\\
&=\big(-i\hspace{1pt}\L_{X_2}+2\hspace{1pt}\widetilde{g_L}f^{1/2}H_1f^{1/2}
-if\L_{X_1}(\widetilde{g_L})\big)\varphi\\
&=-i\big(\L_{X_2}+2\hspace{1pt}\widetilde{g_L}f\L_{X_1}
+\widetilde{g_L}\hspace{1pt}\L_{X_1}(f)+f\L_{X_1}(\widetilde{g_L})\big)\varphi\\
&=\textstyle-i\big(\L_X+\frac12\div_\Omega X\big)\varphi,
\end{align*}
which proves the claim.

(b) If $f\in C^3\big(M;(0,\infty)\big)$, then $X$ is a $C^2$ vector field on the
compact manifold $M$ , and thus $X$ admits a complete flow $\{F_t\}_{t\in\R}$ with
$F_t(p)$ of class $C^2$ in the $p$-variable \cite[Sec.~2.1]{AM78}. For each $t\in\R$,
let $\det_\Omega(F_t)\in C^1(M;\R)$ be the unique function satisfying
$F_t^*\Omega=\det_\Omega(F_t)\Omega$ \cite[Def.~2.5.18]{AM78}. Since $F_0$ is the
identity map, we have $\det_\Omega(F_0)=1$ and thus $\det_\Omega(F_t)>0$ for all
$t\in\R$ by continuity of $F_t(p)$ in the $t$-variable (see \cite[Prop.~2.5.19 \&
2.5.20(ii)]{AM78}). In particular, one can define for each $t\in\R$ the operator
$$
U(t)\hspace{1pt}\varphi
:=\{\det_\Omega(F_t)\}^{1/2}\hspace{1pt}\varphi\circ F_t\hspace{1pt},
\quad\varphi\in C(M).
$$
Some routine computations using \cite[Prop.~2.5.20]{AM78} show that
$\{U(t)\}_{t\in\R}$ defines a strongly continuous unitary group in $\H$ satisfying
$U(t)C^1(M)\subset C^1(M)$ for each $t\in\R$. Thus, it follows from Nelson's theorem
that the generator $D$ of the group $\{U(t)\}_{t\in\R}$ is essentially self-adjoint
on $C^1(M)$. Furthermore, standard computations (see \cite[Sec.~5.4]{AM78}) show that
$$
D\hspace{1pt}\varphi=\textstyle-i\big(\L_X+\frac12\div_\Omega X\big)\varphi
$$
for each $\varphi\in C^1(M)$. This, together with point (a), shows that the operators
$D$ and $A_L$ coincide on $C^1(M)$, and thus that $A_L$ is essentially self-adjoint
on $C^1(M)$.
\end{proof}

\begin{Remark}
We believe it might be possible to prove the essential self-adjointness of the
operator $A_L$ for time changes of class $C^2$, instead of time changes of class
$C^3$ as presented in Lemma \ref{Lem_conj}. Doing so, one would extend all the
results of this section to time changes of class $C^2$, since Lemma \ref{Lem_conj} is
the only instance where a regularity assumption stronger than $C^2$ is needed.
\end{Remark}

We now prove regularity properties of $H$ and $H^2$ with respect to $A_L$.

\begin{Lemma}\label{Lem_reg_A_L}
Let $f\in C^3\big(M;(0,\infty)\big)$, $L>0$ and $z\in\C\setminus\R$. Then,
\begin{enumerate}
\item[(a)] $(H+z)^{-1}\in C^1(A_L)$ with
$$
\big[i(H+z)^{-1},A_L\big]=-(H+z)^{-1}\big(Hg_L+g_LH\big)(H+z)^{-1},
$$

\item[(b)] $(H^2+1)^{-1}\in C^1(A_L)$ with
$$
\big[i(H^2+1)^{-1},A_L\big]
=-(H^2+1)^{-1}\big(H^2g_L+2Hg_LH+g_LH^2\big)(H^2+1)^{-1},
$$

\item[(c)] the multiplication operator $g_L$ satisfies $g_L\in C^1(A_L)$ with
$\big[ig_L,A_L\big]=-\L_X(g_L)$,

\item[(d)] $(H^2+1)^{-1}\in C^2(A_L)$.
\end{enumerate}
\end{Lemma}

\begin{proof}
(a) Let $\varphi\in C^1(M)$. Then, Lemma \ref{lem_regu}(b) implies that
\begin{align*}
&\big\langle\varphi,i(H+z)^{-1}A_L\varphi\big\rangle
-\big\langle A_L\varphi,i(H+z)^{-1}\varphi\big\rangle\\
&=-\frac1L\int_0^L\d t\,\big\langle\e^{-itH}\varphi,
(H+z)^{-1}(Hg+gH)(H+z)^{-1}\e^{-itH}\varphi\big\rangle\\
&=\textstyle-\big\langle\varphi,(H+z)^{-1}
\big\{H\big(\frac1L\int_0^L\d t\,\e^{itH}g\e^{-itH}\big)
+\big(\frac1L\int_0^L\d t\,\e^{itH}g\e^{-itH}\big)H\big\}
(H+z)^{-1}\e^{-itH}\varphi\big\rangle\hspace{1pt}.
\end{align*}
Since
$
\frac1L\int_0^L\d t\,\e^{itH}g\e^{-itH}
=\frac1L\int_0^L\d t\,\big(g\circ\widetilde F_{1,-t}\big)
=g_L
$,
it follows that
$$
\big\langle\varphi,i(H+z)^{-1}A_L\varphi\big\rangle
-\big\langle A_L\varphi,i(H+z)^{-1}\varphi\big\rangle\\
=-\big\langle\varphi,
(H+z)^{-1}\big(Hg_L+g_LH\big)(H+z)^{-1}\varphi\big\rangle\hspace{1pt},
$$
and one concludes using the density of $C^1(M)$ in $\dom(A_L)$.

(b) Let $\varphi\in\H$. Then, it follows from point (a) that
\begin{align*}
&\frac\d{\d t}\hspace{1pt}\e^{-itH_2}(H^2+1)^{-1}\e^{itH_2}\varphi\Big|_{t=0}\\
&=\frac\d{\d t}\hspace{1pt}\e^{-itH_2}(H+i)^{-1}\e^{itH_2}\e^{-itH_2}(H-i)^{-1}
\e^{itH_2}\varphi\Big|_{t=0}\\
&=-(H+i)^{-1}\big(Hg_L+g_LH\big)(H+i)^{-1}(H-i)^{-1}\varphi
-(H+i)^{-1}(H-i)^{-1}\big(Hg_L+g_LH\big)(H-i)^{-1}\varphi\\
&=-(H^2+1)^{-1}\big(H^2g_L+2Hg_LH+g_LH^2\big)(H^2+1)^{-1}\varphi,
\end{align*}
which implies the claim.

(c) Let $\varphi\in C^1(M)$, then we know from Lemma \ref{Lem_conj} that
$$
\textstyle
\big\langle\varphi,ig_LA_L\varphi\big\rangle
-\big\langle A_L\varphi,ig_L\varphi\big\rangle
=\big\langle\varphi,\big[g_L,\L_X+\frac12\div_\Omega X\big]\varphi\big\rangle
=\big\langle\varphi,-\L_X(g_L)\varphi\big\rangle\hspace{1pt}.
$$
Since $\L_X(g_L)\in\linf(M)$, it follows by the density of $C^1(M)$ in $\dom(A_L)$
that $g_L\in C^1(A_L)$ with $[\hspace{1pt}ig_L,A_L]=-\L_X(g_L)$.

(d) Direct computations using point (b) show that
\begin{align*}
&\big[i(H^2+1)^{-1},A_L\big]\\
&=-(H^2+1)^{-1}\big\{(H^2+1)g_L+2(H+i)g_L(H-i)\\
&\qquad\qquad\qquad\qquad+2i(H+i)g_L-2ig_L(H-i)+g_L(H^2+1)\big\}(H^2+1)^{-1}\\
&=-2\re\big\{g_L(H^2+1)^{-1}+2i(H-i)^{-1}g_L(H^2+1)^{-1}
+(H-i)^{-1}g_L(H+i)^{-1}\big\}.
\end{align*}
Morevover, we know from points (a)-(c) that the operators $g_L$, $(H^2+1)^{-1}$,
$(H+i)^{-1}$ and $(H-i)^{-1}$ belong to $C^1(A_L)$. So, one infers from standard
results on the space $C^1(A_L)$ (see \cite[Prop.~5.1.5]{ABG}) that
$\big[i(H^2+1)^{-1},A_L\big]$ also belongs to $C^1(A_L)$.
\end{proof}

In order to apply the theory of Section \ref{Sec_com}, one has to prove at some point
a positive commutator estimate. If the function $f$ were the constant function
$f\equiv1$, then one would have the equalities $H=H_1$,
$$
A_L
=\frac1L\int_0^L\d t\,\e^{itH_1}H_2\e^{-itH_1}
=H_2+\frac1L\int_0^L\d t\,tH_1
=H_2+\frac{LH_1}2\hspace{1pt},
$$
and
$$
\big[iH^2,A_L\big]
=\left[iH_1^2,H_2+\frac{LH_1}2\right]
=2H_1^2
$$
due to \eqref{com_H1H2}. Therefore, one would immediately obtain a strict Mourre
estimate for $H^2\equiv H_1^2$. This suggests to study the positivity of the
commutator $\big[iH^2,A_L\big]$ also in the case $f\not\equiv1$. A glimpse at Lemma
\ref{Lem_reg_A_L}(b) tells us that $\big[iH^2,A_L\big]$ is equal to the operator
$H^2g_L+2Hg_LH+g_LH^2$, which does not seem to exhibit any explicit positivity.
However, if the function $g_L$ were positive, then all the operators $g_L$, $H^2$ and
$Hg_LH$ would be positive, and thus the sum $H^2g_L+2Hg_LH+g_LH^2$ would be more
likely to be positive as a whole. In fact, this is exactly what happens and this was
the whole point of choosing the conjugate operator $A_L$ as we did. Thanks to the
unique ergodicity of the horocycle flow, one has $g_L>0$ if $L>0$ is big enough and
the operator $H^2$ satisfies a strict Mourre estimate with respect to
$A_L$\hspace{1pt}:

\begin{Lemma}[Strict Mourre estimate for $H^2$]\label{M_estimate}
Let $f\in C^3\big(M;(0,\infty)\big)$ and take $L>0$ big enough. Then, $g_L>0$
and one has for each bounded Borel set $J\subset(0,\infty)$ that
$$
E^{H^2}(J)\big[iH^2,A_L\big]E^{H^2}(J)\ge aE^{H^2}(J)\quad\hbox{with}\quad
a:=2\hspace{1pt}\inf(J)\cdot\inf_{p\in M}g_L(p)>0.
$$
\end{Lemma}

\begin{proof}
(i) The horocycle flow $\{F_{1,t}\}_{t\in\R}$ is uniquely ergodic with respect to the
measure $\mu_\Omega$ \cite{Fur73}. Therefore, we know from the theory of time changes
on compact metric spaces \cite[Prop.~3]{Hum74} that the flow
$\{\widetilde F_{1,t}\}_{t\in\R}$ is also uniquely ergodic with respect to the
measure
$$
\d\widetilde\mu_\Omega
:=\frac{f^{-1}\d\mu_\Omega}{\int_Mf^{-1}\hspace{1pt}\d\mu_\Omega}\hspace{1pt}.
$$
It follows that (see \cite[Prop.~1.3.4]{Qui07})
\begin{align*}
\lim_{L\to\infty}g_L
=\lim_{L\to\infty}\bigg(\frac12
-\frac1{2L}\int_0^L\d t\,\L_{X_2}\big(\ln(f)\big)\circ\widetilde F_{1,-t}\bigg)
&=\frac12-\frac12\int_M\d\widetilde\mu_\Omega\,\L_{X_2}\big(\ln(f)\big)\\
&=\frac12+\frac1{2\int_Mf^{-1}\hspace{1pt}\d\mu_\Omega}\int_M\d\mu_\Omega\,
\L_{X_2}\big(f^{-1}\big)\\
&=\frac12+\frac i{2\int_Mf^{-1}\hspace{1pt}\d\mu_\Omega}\hspace{1pt}	
\big\langle1,H_2f^{-1}\big\rangle\\
&=\frac12
\end{align*}
uniformly on $M$. Thus, $g_L>0$ if $L>0$ is big enough.

(ii) We know from Equation \eqref{2com} and Lemma \ref{Lem_reg_A_L}(b) that
$$
E^{H^2}(J)\big[iH^2,A_L\big]E^{H^2}(J)
=E^{H^2}(J)\big(H^2g_L+2Hg_LH+g_LH^2\big)E^{H^2}(J).
$$
But, point (i) implies that
$$
E^{H^2}(J)2Hg_LHE^{H^2}(J)\ge aE^{H^2}(J)
\quad\hbox{with}\quad a=2\inf(J)\cdot\inf_{p\in M}g_L(p)>0.
$$
Therefore, it is sufficient to show that
$E^{H^2}(J)\big(H^2g_L+g_LH^2\big)E^{H^2}(J)\ge0$.

So, for any $\varepsilon>0$ let
$H^2_\varepsilon:=H^2\big(\varepsilon^2H^2+1\big)^{-1}$ and
$H^\pm_\varepsilon:=H(\varepsilon H\pm i)^{-1}$. Then, the inclusion
$g_L^{1/2}\in C^1(H)$ (which can be proved as in Lemma \ref{Lem_reg_A_L}(c)) implies
that
$$
\slim_{\varepsilon\searrow0}\big[H^\pm_\varepsilon,g_L^{1/2}\big]
=\pm\slim_{\varepsilon\searrow0}
(\varepsilon H\pm i)^{-1}\big[iH,g_L^{1/2}\big](\varepsilon H\pm i)^{-1}\\
=\pm i\big[g_L^{1/2},H\big].
$$
Therefore, for each $\varphi\in\H$ it follows that
\begin{align*}
&\big\langle\varphi, E^{H^2}(J)\big(H^2g_L+g_LH^2\big)E^{H^2}(J)\varphi\big\rangle\\
&=\lim_{\varepsilon\searrow0}\big\langle\varphi,E^{H^2}(J)
\big(H^2_\varepsilon g_L^{1/2}g_L^{1/2}+g_L^{1/2}g_L^{1/2}H^2_\varepsilon\big)
E^{H^2}(J)\varphi\big\rangle\\
&=\lim_{\varepsilon\searrow0}\big\langle\varphi,
E^{H^2}(J)\big(\big[H^2_\varepsilon,g_L^{1/2}\big]g_L^{1/2}
+2\hspace{1pt}g_L^{1/2}H^2_\varepsilon g_L^{1/2}
+g_L^{1/2}\big[g_L^{1/2},H^2_\varepsilon\big]\big)E^{H^2}(J)\varphi\big\rangle\\
&\ge\lim_{\varepsilon\searrow0}\big\langle\varphi,
E^{H^2}(J)\big(\big[H^2_\varepsilon,g_L^{1/2}\big]g_L^{1/2}
+g_L^{1/2}\big[g_L^{1/2},H^2_\varepsilon\big]\big)E^{H^2}(J)\varphi\big\rangle\\
&=\lim_{\varepsilon\searrow0}\big\langle\varphi,
E^{H^2}(J)\big(H^+_\varepsilon\big[H^-_\varepsilon,g_L^{1/2}\big]g_L^{1/2}
+\big[H^+_\varepsilon,g_L^{1/2}\big]H^-_\varepsilon g_L^{1/2}\\
&\hspace{90pt}+g_L^{1/2}\big[g_L^{1/2},H^+_\varepsilon\big]H^-_\varepsilon
+g_L^{1/2}H^+_\varepsilon\big[g_L^{1/2},H^-_\varepsilon\big]\big)E^{H^2}(J)
\varphi\big\rangle\\
&=\lim_{\varepsilon\searrow0}\big\langle\varphi,
E^{H^2}(J)\big(H\big[H,g_L^{1/2}\big]g_L^{1/2}
+\big[H^+_\varepsilon,g_L^{1/2}\big]g_L^{1/2}H^-_\varepsilon
+\big[H^+_\varepsilon,g_L^{1/2}\big]\big[H^-_\varepsilon,g_L^{1/2}\big]\\
&\hspace{90pt}+g_L^{1/2}\big[g_L^{1/2},H\big]H
+H^+_\varepsilon g_L^{1/2}\big[g_L^{1/2},H^-_\varepsilon\big]
+\big[g_L^{1/2},H^+_\varepsilon\big]\big[g_L^{1/2},H^-_\varepsilon\big]\big)
E^{H^2}(J)\varphi\big\rangle\\
&=\big\langle\varphi,E^{H^2}(J)\big(H\big[H,g_L^{1/2}\big]g_L^{1/2}
+\big[H,g_L^{1/2}\big]g_L^{1/2}H
+2\big[H,g_L^{1/2}\big]^2+g_L^{1/2}\big[g_L^{1/2},H\big]H\\
&\hspace{260pt}+Hg_L^{1/2}\big[g_L^{1/2},H\big]\big)E^{H^2}(J)\varphi\big\rangle\\
&=\big\langle\varphi,E^{H^2}(J)2\big[H,g_L^{1/2}\big]^2E^{H^2}(J)
\varphi\big\rangle\\
&\ge0,
\end{align*}
which implies the claim.
\end{proof}

Using the previous results for $H^2$, one can finally determine the structure of the
spectrum of $H$ (and thus that of $\widetilde H$)\hspace{1pt}:

\begin{Theorem}[Spectral properties of $H$]\label{thm_spec}
Let $f\in C^3\big(M;(0,\infty)\big)$. Then, $H$ has purely absolutely continuous
spectrum, except at $\hspace{1pt}0$, where it has a simple eigenvalue with eigenspace
$\C\cdot f^{-1/2}$. In particular, the self-adjoint operator $\widetilde H$
associated to the vector field $fX_1$ has purely absolutely continuous spectrum,
except at $\hspace{1pt}0$, where it has a simple eigenvalue with eigenspace
$\C\cdot1$.
\end{Theorem}

\begin{proof}
We know from Lemmas \ref{Lem_reg_A_L}(d) and \ref{M_estimate} that
$(H^2+1)^{-1}\in C^2(A_L)$ and that $H^2$ satisfies a strict Mourre estimate on each
bounded Borel subset of $(0,\infty)$. It follows by Theorem \ref{thm_abs} that $H^2$
has purely absolutely continuous spectrum, except at $0$, where it may have an
eigenvalue. Accordingly, the Hilbert space $\H$ admits the orthogonal decomposition
$$
\H=\ker(H^2)\oplus\H_{\rm ac}(H^2),
$$
with $\H_{\rm ac}(H^2)$ the subspace of absolute continuity of $H^2$.

Now, the function $\lambda\mapsto\lambda^2$ has the Luzin N property on
$\R$\hspace{1pt}; namely, if $J$ is a Borel subset of $\R$ with Lebesgue measure
zero, then $J^2$ also has Lebesgue measure zero. It follows that
$\H_{\rm ac}(H^2)\subset\H_{\rm ac}(H)$, with $\H_{\rm ac}(H)$ the subspace of
absolute continuity of $H$ (see Proposition 29, Section 3.5.4 of \cite{BW83}).
Furthermore, we have that
$$
\ker(H^2)
=\ker(H)
=\U^*\ker(\widetilde H)
=\C\cdot f^{-1/2}
$$
due to the equality $H=\U^*\widetilde H\U$ and the ergodicity of the flow
$\{\widetilde F_{1,t}\}_{t\in\R}$. We thus infer that
$$
\H=\ker(H^2)\oplus\H_{\rm ac}(H^2)\subset\ker(H)\oplus\H_{\rm ac}(H).
$$
So, one necessarily has $\H=\ker(H)\oplus\H_{\rm ac}(H)$, meaning that $H$ has purely
absolutely continuous spectrum, except at $\hspace{1pt}0$, where it has a simple
eigenvalue with eigenspace $\ker(H)\equiv\C\cdot f^{-1/2}$. Since
$H=\U^*\widetilde H\U$, this implies that $\widetilde H$ has purely absolutely
continuous spectrum, except at $\hspace{1pt}0$, where it has a simple eigenvalue with
eigenspace $\C\cdot1$.
\end{proof}

Theorem \ref{thm_spec} establishes the absolute continuity of time changes of
horocycle flows on compact surfaces of constant negative curvature for time changes
of class $C^3$. This improves Theorem 6 of \cite{FU12}, where G. Forni and
C. Ulcigrai show the same result for time changes in a Sobolev space of order
$>11/2$ (under the same assumption, the authors of \cite{FU12} also show that the
maximal spectral type is equivalent to Lebesgue). This also complements Theorem 4.2
of \cite{Tie12}, where the absolute continuity is shown for surfaces of finite volume
and time changes of class $C^2$ under the additional condition of A. G. Kushnirenko.

We note that it would be interesting to see if the technics of this section could be
adapted to the case of horocycle flows on surfaces of finite volume or surfaces of
non-constant negative curvature. In the first case, one would have to deal with the
fact that the horocycle flow is not uniquely ergodic (see \cite{Dan81,DS84}), while
in the second case one would have to deal with the fact that the horocycle flow is
uniquely ergodic, but with the Margulis parametrisation and with respect to the
Bowen-Margulis measure (see \cite{Cou09,Mar75}).

%--------------------------------------------------------------------------------------
\section{Skew products over translations}\label{sec_skew}
\setcounter{equation}{0}
%--------------------------------------------------------------------------------------

Let $X$ be a compact metric abelian Banach Lie group with normalised Haar measure
$\mu$ (such a group is isomorphic to a subgroup of the torus $\T^{\aleph_0}$, see
\cite[Thm.~8.45]{HM06}). Take $\{y_t\}_{t\in\R}$ a $C^1$ one-parameter subgroup of
$X$ and let $\{F_t\}_{t\in\R}$ be the corresponding translation flow, \ie,
$$
F_t(x):=y_t\hspace{1pt}x,\quad t\in\R,~x\in X.
$$
Assume that the translation $F_1$ is ergodic (so that both flows
$\{F_\ell\}_{\ell\in\Z}$ and $\{F_t\}_{t\in\R}$ are uniquely ergodic, see
\cite[Thm.~4.1.1]{CFS82} and \cite[Sec.~1.2.2]{Kre85}) and associate to
$\{F_t\}_{t\in\R}$ the operators
$$
V_t\hspace{1pt}\varphi:=\varphi\circ F_t\hspace{1pt},\quad t\in\R,~\varphi\in C(X).
$$
Due to the continuity of the map $\R\ni t\mapsto y_t\in X$ and the smoothness of the
group operation, the family $\{V_t\}_{t\in\R}$ extends to a strongly continuous
unitary group in $\H:=\ltwo(X,\mu)$ satisfying
$V_t\hspace{1pt}C^\infty(X)\subset C^\infty(X)$ for each $t\in\R$. It follows from
Nelson's theorem \cite[Prop.~5.3]{Amr09} that the generator of the group
$\{V_t\}_{t\in\R}$
$$
H\varphi:=\slim_{t\to0}it^{-1}\big(V_t-1\big)\varphi,
\quad\varphi\in\dom(H):=\left\{\varphi\in\H\mid
\lim_{t\to0}|t|^{-1}\big\|\big(V_t-1\big)\varphi\big\|<\infty\right\},
$$
is essentially self-adjoint on $C^\infty(X)$, and one has
$$
H\varphi:=-i\hspace{1pt}\L_Y\varphi,\quad\varphi\in C^\infty(M),
$$
with $Y$ the $C^0$ divergence-free vector field associated to $\{F_t\}_{t\in\R}$ and
$\L_Y$ the corresponding Lie derivative. Furthermore, the operator $V_1$ has pure
point spectrum
$\sigma(V_1)=\overline{\big\{\gamma(y_1)\mid\gamma\in\widehat X\big\}}$, with
$\widehat X$ the character group of $X$ (see \cite[Thm.~3.5]{Wal82}).

Now, let $G$ be a compact metric abelian group with Haar measure $\nu$ and character
group $\widehat G$, and let $\phi:X\to G$ be a measurable function (a cocycle). Then,
one can define the skew product $T:X\times G\to X\times G$ given by 
$T(x,z):=\big(y_1\hspace{1pt}x,\phi(x)\hspace{1pt}z\big)$ and the corresponding
unitary operator
\begin{equation}\label{def_W}
W\hspace{1pt}\psi:=\psi\circ T\hspace{1pt},\quad\psi\in\ltwo(X\times G,\mu\times\nu).
\end{equation}
It is known \cite[Sec.~3.1]{Goo99} that the operator $W$ is reduced by the
orthogonal decomposition
$$
\ltwo(X\times G,\mu\times\nu)
=\bigoplus_{\chi\in\widehat G}\hspace{1pt}L_\chi\hspace{1pt},
\quad L_\chi:=\big\{\varphi\otimes\chi\mid\varphi\in\H\big\},
$$
and that the restriction $W|_{L_\chi}$ is unitarily equivalent to the unitary
operator
$$
U_\chi\hspace{1pt}\varphi:=(\chi\circ\phi)\hspace{1pt}V_1\hspace{1pt}\varphi,
\quad\varphi\in\H.
$$
Furthermore, the operator $U_\chi$ satisfies the following purity
law: the spectrum of $U_\chi$ has
uniform multiplicity and is either purely punctual, purely singularly continuous or
purely Lebesgue (this follows from Helson's analysis \cite{Hel86}; see
\cite[Thm.~2]{Gro91}, \cite[Thm.~4]{GL99} and \cite[p.~8560]{Lem09}).

In the sequel, we treat the case where the cocycle $\phi$ satisfies the following
assumption\hspace{1pt}:

\begin{Assumption}[Cocycle]\label{ass_phi}
The map $\phi:X\to G$ satisfies $\phi=\xi\hspace{1pt}\eta$, where
\begin{enumerate}
\item[(i)] $\xi:X\to G$ is a continuous group homomorphism,
\item[(ii)] $\eta\in C(X;G)$ has a Lie derivative $\L_Y(\chi\circ\eta)$ which
satisfies a uniform Dini condition along the flow $\{F_t\}_{t\in\R}$, \ie,
$$
\int_0^1\frac{\d t}t\,\big\|\L_Y(\chi\circ\eta)\circ F_t
-\L_Y(\chi\circ\eta)\big\|_{\linf(X)}
=\int_0^1\frac{\d t}t\,\big\|V_t\hspace{1pt}\L_Y(\chi\circ\eta)V_{-t}
-\L_Y(\chi\circ\eta)\big\|_{\B(\H)}<\infty.
$$
\end{enumerate}
\end{Assumption}

We start the analysis with a first lemma on the regularity of the operators $U_\chi$.
We use the fact that the map $\R\ni t\mapsto(\chi\circ\xi)(y_t)\in\mathbb S^1$ is a
character on $\R$, and thus of class $C^\infty$. We also use the notations
$$
\xi_0:=\frac\d{\d t}\hspace{1pt}(\chi\circ\xi)(y_t)\Big|_{t=0}\hspace{1pt},\qquad
g:=|\xi_0|^2-\xi_0\hspace{1pt}\frac{\L_Y(\chi\circ\eta)}{\chi\circ\eta}
\qquad\hbox{and}\qquad
A:=-i\hspace{1pt}\xi_0 H,
$$
and observe that $\xi_0\in i\hspace{1pt}\R$, that $g:X\to\R$ satisfies a uniform Dini
condition along $\{F_t\}_{t\in\R}$, that $A$ is self-adjoint with
$\dom(A)\supset\dom(H)$ and that $\xi_0$, $g$ and $A$ depend on $\chi$ (even if we do
not specify it in the notation).

\begin{Lemma}\label{UtetA}
Let $\phi$ satisfy Assumption \ref{ass_phi}. Then, $U_\chi\in C^{1+0}(A)$ with
$\big[A,U_\chi\big]=g\hspace{1pt}U_\chi$\hspace{1pt}.
\end{Lemma}

\begin{proof}
Since $A$ and $V_1$ commute, one has for each $\varphi\in C^\infty(X)$ that
\begin{align*}
\big\langle A\hspace{1pt}\varphi,U_\chi\hspace{1pt}\varphi\big\rangle
-\big\langle\varphi,U_\chi\hspace{1pt}A\hspace{1pt}\varphi\big\rangle
=\big\langle\varphi,\big[A,\chi\circ\phi\big]V_1\hspace{1pt}\varphi\big\rangle
=\big\langle\varphi,-\xi_0\hspace{1pt}\L_Y(\chi\circ\phi)\hspace{1pt}V_1\hspace{1pt}
\varphi\big\rangle\hspace{1pt}.
\end{align*}
Furthermore, the homomorphism property of $\chi$ and $\xi$ and the Leibniz rule for
Lie derivatives imply that
$$
\L_Y(\chi\circ\phi)
=\L_Y(\chi\circ\xi)\hspace{1pt}(\chi\circ\eta)
+(\chi\circ\xi)\hspace{1pt}\L_Y(\chi\circ\eta)
=\left(\xi_0+\frac{\L_Y(\chi\circ\eta)}{\chi\circ\eta}\right)(\chi\circ\phi).
$$
It follows that
$$
\big\langle A\hspace{1pt}\varphi,U_\chi\hspace{1pt}\varphi\big\rangle
-\big\langle\varphi,U_\chi\hspace{1pt}A\hspace{1pt}\varphi\big\rangle
=\big\langle\varphi,g\hspace{1pt}U_\chi\hspace{1pt}\varphi\big\rangle,
$$
with $g\in\linf(X)$. So, one has $U_\chi\in C^1(A)$ with
$[A,U_\chi]=g\hspace{1pt}U_\chi$ due to the density of $C^\infty(X)$ in $\dom(A)$.

To show that $U_\chi\in C^{1+0}(A)$, one has to check that
$
\int_0^1\frac{\d t}t\hspace{1pt}
\big\|\e^{-itA}[A,U_\chi]\e^{itA}-[A,U_\chi]\big\|_{\B(\H)}<\infty
$.
But since $[A,U_\chi]=g\hspace{1pt}U_\chi$ with $U_\chi\in C^1(A)$, one is reduced
to showing that
$$
\int_0^1\frac{\d t}t\hspace{1pt}\big\|\e^{-itA}g\e^{itA}-g\big\|_{\B(\H)}<\infty
~\iff~\int_0^{-i\hspace{1pt}\xi_0}\frac{\d s}s\,
\big\|V_s\hspace{2pt}gV_{-s}-g\big\|_{\B(\H)}<\infty,
$$
which is is readily verified due to the uniform Dini condition satisfied by $g$.
\end{proof}

Since $U_\chi\in C^1(A)$, we know from Section \ref{Sec_com} that the operator
$$
A_n\hspace{1pt}\varphi
:=\frac1n\sum_{\ell=0}^{n-1}
U_\chi^{-\ell}\hspace{1pt}A\hspace{1pt}U_\chi^\ell\hspace{1pt}\varphi
=\frac1n\sum_{\ell=0}^{n-1}U_\chi^{-\ell}\big[A,U_\chi^\ell\big]\varphi
+A\hspace{1pt}\varphi,\quad n\in\N^*,~\varphi\in\dom(A_n):=\dom(A),
$$
is self-adjoint. In the next lemma, we prove regularity properties of the operator
$U_\chi$ with respect to $A_n$ and the strict Mourre estimate for $U_\chi$.
The averages of the function $g$ along the flow $\{F_\ell\}_{\ell\in\Z}$, \ie,
$$
g_n:=\frac1n\sum_{\ell=0}^{n-1}g\circ F_{-\ell}\hspace{1pt},\quad n\in\N^*,
$$
appear in a natural way.

\begin{Lemma}[Strict Mourre estimate for $U_\chi$]\label{Mourre_U_chi}
Let $\phi$ satisfy Assumption \ref{ass_phi} with $\chi\circ\xi\not\equiv1$, and
suppose $F_1$ is ergodic. Then,
\begin{enumerate}
\item[(a)] one has $U_\chi\in C^{1+0}(A_n)$ with $[A_n,U_\chi]=g_nU_\chi$,
\item[(b)] if $n$ is big enough, one has $g_n>0$ and
$(U_\chi)^*\big[A_n,U_\chi\big]\ge a$\hspace{1pt} with $a:=\inf_{x\in X}g_n(x)>0$.
\end{enumerate}
\end{Lemma}

\begin{proof}
(a) We know from Lemma \ref{UtetA} that $U_\chi\in C^{1+0}(A)$. So, it follows from
the abstract result \cite[Lemma~4.1]{FRT12} that $U_\chi\in C^{1+0}(A_n)$ with
$
[A_n,U_\chi]
=\frac1n\sum_{\ell=0}^{n-1}U_\chi^{-\ell}\hspace{1pt}[A,U_\chi]\hspace{1pt}
U_\chi^\ell
$.
Using the equality $[A,U_\chi]=g\hspace{1pt}U_\chi$, one thus obtains that
\begin{align*}
\big[A_n,U_\chi\big]
=\left(\frac1n\sum_{\ell=0}^{n-1}U_\chi^{-\ell}\hspace{1pt}g\hspace{1pt}
U_\chi^\ell\right)U_\chi
=\left(\frac1n\sum_{\ell=0}^{n-1}g\circ F_{-\ell}\right)U_\chi
=g_n\hspace{1pt}U_\chi,
\end{align*}
which concludes the proof of the claim.

(b) Due to the unique ergodicity of the discrete flow $\{F_\ell\}_{\ell\in\Z}$, we
know that $\lim_{n\to\infty}g_n=\int_X\d\mu\,g$ uniformly on $X$. Using the fact that
$\chi\circ\eta=\e^{if_{\chi,\eta}}$ for some real function $f_{\chi,\eta}\in\dom(H)$,
we thus deduce that
$$
\lim_{n\to\infty}g_n
=\int_X\d\mu\,g
=|\xi_0|^2-\xi_0\int_X\d\mu\,\frac{\L_Y(\chi\circ\eta)}{\chi\circ\eta}
=|\xi_0|^2+\xi_0\hspace{1pt}\big\langle1,Hf_{\chi,\eta}\big\rangle
=|\xi_0|^2
$$
uniformly on $X$. But, since the character $\chi\circ\xi:X\to\mathbb S^1$ is
nontrivial, we know that $\xi_0\neq0$ due to the unique ergodicity of the continuous
flow $\{F_t\}_{t\in\R}$ (see \cite[Thm.~4.1.1']{CFS82}). Therefore, $g_n>0$ if $n>0$
is big enough, and point (a) implies that
$$
(U_\chi)^*\big[A_n,U_\chi\big]=(U_\chi)^*g_n\hspace{1pt}U_\chi\ge a,
$$
with $a=\inf_{x\in X}g_n(x)>0$.
\end{proof}

Using what precedes, we can determine the spectral properties of the operators
$U_\chi$ and $W$ (see \eqref{def_W} for the definition of $W$)\hspace{1pt}:

\begin{Theorem}[Spectral properties of $U_\chi$ and $W$]\label{spec_co}
Let $\phi$ satisfy Assumption \ref{ass_phi} with $\chi\circ\xi\not\equiv1$, and
suppose that $F_1$ is ergodic. Then, the operator $U_\chi$ has purely Lebesgue
spectrum. In particular, the restriction of $W$ to the subspace
$
\bigoplus_{\chi\in\widehat G,\,\chi\circ\xi\not\equiv1}L_\chi
\subset\ltwo(X\times G,\mu\times\nu)
$
has countable Lebesgue spectrum.
\end{Theorem}

\begin{proof}
We know from Lemma \ref{Mourre_U_chi} that $U_\chi\in C^{1+0}(A_n)$ and that $U_\chi$
satisfies a strict Mourre estimate on all of $\mathbb S^1$. It follows by Theorem
\ref{thm_unitary} that $U_\chi$ has a purely absolutely continuous spectrum, and thus
has a purely Lebesgue spectrum due to the purity law. The claim on $W$ follows from
what precedes if one takes into account the separability of the Hilbert space
$\ltwo(X\times G,\mu\times\nu)$.
\end{proof}

Theorem \ref{spec_co} provides a general criterion for the presence of countable
Lebesgue spectrum for skew products over translations. In the particular case where
$X=\T^d\simeq\R^d/\Z^d$ and $G=\T^{d'}\simeq\R^{d'}/\Z^{d'}$ for some $d,d'\ge1$, the
flow $\{F_t\}_{t\in\R}$ is given (in additive notation) by
$$
F_t(x):=ty+x~\hbox{(mod $\Z^d$)},\quad t\in\R,~x\in\T^d,
$$
for some $y:=(y_1,y_2,\ldots,y_d)\in\R^d$. So, one has $\L_Y=y\cdot\nabla$, and $F_1$
is ergodic if and only if the numbers $y_1,y_2,\ldots,y_d,1$ are rationally
independent \cite[Sec.~3.1]{CFS82}. Furthermore, each group homomorphism
$\xi:\T^d\to\T^{d'}$ is given by $\xi(x):=Nx$ (mod $\Z^{d'}$) for some $d'\times d$
matrix $N$ with integer entries, and each character $\chi_m\in\widehat{\T^{d'}}$ is
given by $\chi_m(z):=\e^{2\pi im\cdot z}$ for some $m\in\Z^{d'}$. Therefore,
$$
\chi_m\circ\xi\not\equiv1
\iff\e^{2\pi im\cdot Nx}\neq1~~\hbox{for some $x\in\T^d$}
\iff N^{\sf T}m\neq0\in\Z^d,
$$
and we obtain the following corollary of Theorem \ref{spec_co}.

\begin{Corollary}[The case of tori]\label{cor_tori}
Let $y_1,y_2,\ldots,y_d,1\in\R$ be rationally independent, let
$\chi_m\in\widehat{\T^{d'}}$ be given by $\chi_m(z):=\e^{2\pi im\cdot z}$ for some
$m\in\Z^{d'}$ and let $\phi:\T^d\to\T^{d'}$ satisfy $\phi=\xi+\eta$, where
\begin{enumerate}
\item[(i)] $\xi:\T^d\to\T^{d'}$ is given by $\xi(x):=Nx$ (\hspace{1pt}mod $\Z^{d'}$)
for some $d'\times d$ matrix $N$ with integer entries,
\item[(ii)] $\eta\in C(\T^d;\T^{d'})$ has a derivative $y\cdot\nabla(m\cdot\eta)$
which satisfies the uniform Dini condition
$$
\int_0^1\frac{\d t}t\,\big\|\;\!y\cdot\nabla(m\cdot\eta)\circ F_t
-y\cdot\nabla(m\cdot\eta)\big\|_{\linf(X)}
<\infty.
$$
\end{enumerate}
Assume also that $N^{\sf T}m\neq0$. Then, the operator
$U_{\chi_m}\equiv\e^{2\pi im\cdot(\xi+\eta)}V_1$ has purely Lebesgue spectrum. In
particular, the restriction of $W$ to the subspace
$\bigoplus_{m\in\Z^{d'}\!,N^{\sf T}m\neq0}L_{\chi_m}$ has countable Lebesgue
spectrum.
\end{Corollary}

In the case $d=d'=1$, Corollary \ref{cor_tori} implies that the restriction of $W$ to
$\bigoplus_{m\in\Z\setminus\{0\}}L_{\chi_m}$ has countable Lebesgue spectrum if
$\phi(x)=Nx+\eta(x)$, with $N\in\Z\setminus\{0\}$ and with $\eta\in C^1(\T;\T)$ such
that $\eta'$ is Dini-continuous. This is a bit more restrictive than Theorem 1 of
\cite{ILR93}, where A. Iwanik, M. Lema\'nzyk and D. Rudolph show the same result under
the condition that $\eta$ is absolutely continuous with $\eta'$ of bounded variation
(see also \cite[Sec.~2]{Iwa97_1} for another sufficient condition given in terms of
the Fourier coefficients of $\eta$). Results prior to \cite{ILR93} along this line
can be found in the papers of A. G. Kushnirenko \cite{Kus74} and G. H. Choe
\cite{Cho87}.

In the case $d,d'\ge1$, Corollary \ref{cor_tori} implies that the restriction of $W$
to $\bigoplus_{m\in\Z^{d'}\!,N^{\sf T}m\neq0}L_{\chi_m}$ has countable Lebesgue
spectrum if $\phi(x)=Nx+\eta(x)$, with $N$ a $d'\times d$ matrix with integer entries
and with $\eta\in C(\T^d;\T^{d'})$ such that $y\cdot\nabla(m\cdot\eta)$ exists and
satisfies a uniform Dini condition along the flow $\{F_t\}_{t\in\R}$. This
complements the result of \cite[Sec.~3]{Iwa97_2}, where A. Iwanik shows the same
result for functions $\eta\in C^1(\T^d;\T^{d'})$ with Fourier coefficients satisfying
some decay assumption (see also the works of B. Fayad \cite{Fay02} and K.
Fr{\c{a}}czek \cite{Fra95,Fra00} for related results on the spectrum of skew products
on tori).

We note that it would be interesting to see if the technics of this section could be
adapted to the case of cocycles taking values in non-abelian groups, such as the case
of ${\sf SU}(2)$ considered in \cite{Fra00_2}.

%--------------------------------------------------------------------------------------
\section{Furstenberg transformations}\label{Sec_Fur}
\setcounter{equation}{0}
%--------------------------------------------------------------------------------------

For each integer $n\ge1$, we denote by $\mu_n$ the normalised Haar measure on the
torus $\T^n\simeq\R^n/\Z^n$ and we set $\H_n:=\ltwo\big(\T^n,\mu_n\big)$ for the
corresponding Hilbert space. Furstenberg transformations \cite[Sec.~2.3]{Fur61} are
the invertible measure-preserving maps $T_d:\T^d\to\T^d$ ($d\ge2$) given by
\begin{align*}
&T_d(x_1,x_2,\ldots,x_d)\\
&:=\big(x_1+y,x_2+b_{2,1}x_1+h_1(x_1),\ldots,
x_d+b_{d,1}x_1+\cdots+b_{d,d-1}x_{d-1}+h_{d-1}(x_1,x_2,\ldots,x_{d-1})\big)
~\hbox{(mod $\Z^d$)},
\end{align*}
where $y\in\R\setminus\Q$, $b_{j,k}\in\Z$, $b_{\ell,\ell-1}\ne0$ for
$\ell\in\{2,\ldots,d\}$ and each $h_j:\T^j\to\R$ satisfies a uniform Lipschitz
condition in the variable $x_j$. The corresponding Koopman operator
\begin{equation}\label{def_W_d}
W_d:\H_d\to\H_d\hspace{1pt},\quad\varphi\mapsto\varphi\circ T_d\hspace{1pt},
\end{equation}
is reduced by the orthogonal decompositions
\begin{equation}\label{dec_ortho}
\H_d
=\H_1\oplus\bigoplus_{j\in\{2,\ldots,d\}}\big(\H_j\cap \H_{j-1}^\perp\big)
=\H_1\oplus\bigoplus_{j\in\{2,\ldots,d\},\,k\in\Z\setminus\{0\}}\H_{j,k}\hspace{1pt},
\end{equation}
where the subspaces $\H_{j,k}\subset\H_j$ are defined by
$
\H_{j,k}:=\overline{{\rm Span}\big\{\eta\otimes\chi_k\mid\eta\in\H_{j-1}\big)\big\}}
$,
with $\chi_k\in\widehat\T$ the character given by $\chi_k(x_j):=\e^{2\pi ikx_j}$ (see
\cite[Sec.~13.3]{CFS82} for details). Furthermore, the restriction $W_d|_{\H_{j,k}}$
is unitarily equivalent to the unitary operator given by
\begin{equation}\label{Ujk}
U_{j,k}\hspace{1pt}\eta:=\e^{2\pi ik\phi_j}W_{j-1}\hspace{1pt}\eta,
\quad\eta\in\H_{j-1},
\end{equation}
with
$
\phi_j(x_1,x_2,\ldots,x_{j-1})
:=b_{j,1}x_1+\cdots+b_{j,j-1}x_{j-1}+h_{j-1}(x_1,x_2,\ldots,x_{j-1})
$.

The operators $U_{j,k}=\e^{2\pi ik\phi_j}W_{j-1}$ are similar to the operators
$U_\chi=(\chi\circ\phi)\hspace{1pt}V_1$ studied in Section \ref{sec_skew}. So, we
apply to them the same method. First, we define an operator (vector field) $A$ which
commutes with $W_{j-1}$ and has an appropriate commutator with $\e^{2\pi ik\phi_j}$,
and then we use as a congugate operator the average  
$\frac1n\sum_{\ell=0}^{n-1}U_{j,k}^{-\ell}\hspace{1pt}A\hspace{1pt}U_{j,k}^\ell$ of
$A$ along the flow $\big\{U_{j,k}^\ell\big\}_{\ell\in\Z}$ generated by $U_{j,k}$.

We start with the definition of the operator $A$ and then we prove regularity
properties of the operators $U_{j,k}$ with respect to $A$. For this, we recall that
the translation group $\{V_{t,{j-1}}\}_{t\in\R}$ in $\H_{j-1}$ given by
$$
\big(V_{t,{j-1}}\eta\big)(x_1,x_2,\ldots,x_{j-1})
:=\eta\big(x_1,x_2,\ldots,x_{j-1}-t\hbox{ (mod $\Z$)}\big),
\quad t\in\R,~\eta\in C(\T^{j-1}),
$$
has self-adjoint generator $P_{j-1}:=-i\hspace{1pt}\partial_{j-1}$ which is
essentially self-adjoint on $C^\infty(\T^{j-1})$. Also, for $j\in\{2,\ldots,d\}$ and
$k\in\Z\setminus\{0\}$, we use the notations
$$
g:=1+(b_{j,j-1})^{-1}\partial_{j-1}h_{j-1}
\qquad\hbox{and}\qquad
A:=(2\pi kb_{j,j-1})^{-1}P_{j-1},
$$
and observe that $A$ is self-adjoint with $\dom(A)=\dom(P_{j-1})$ and that
$g\in\linf(\T^{j-1})$ due to the uniform Lipschitz condition satisfied by $h_{j-1}$
in the variable $x_{j-1}$. We also note that $g$ and $A$ depend on $j$ and $k$, even
if we do not specify it in the notation.

\begin{Lemma}\label{UetA}
Let $j\in\{2,\ldots,d\}$ and $k\in\Z\setminus\{0\}$. Assume that $h_{j-1}$ is of
class $C^1$ in the variable $x_{j-1}$ and that $\partial_{j-1}h_{j-1}$ satisfies a
uniform Dini condition in the variable $x_{j-1}$. Then, one has
$U_{j,k}\in C^{1+0}(A)$ with $[A,U_{j,k}]=g\hspace{1pt}U_{j,k}$\hspace{1pt}.
\end{Lemma}

\begin{proof}
Since $A$ and $W_{j-1}$ commute and since $h_{j-1}$ satisfies a uniform Lipschitz
condition in the variable $x_{j-1}$, one has for each $\eta\in C^\infty(\T^{j-1})$
that
$$
\big\langle A\hspace{1pt}\eta,U_{j,k}\hspace{1pt}\eta\big\rangle_{\H_{j-1}}
-\big\langle\eta,U_{j,k}\hspace{1pt}A\hspace{1pt}\eta\big\rangle_{\H_{j-1}}
=\big\langle\eta,
\big[A,\e^{2\pi ik\phi_j}\big]W_{j-1}\eta\big\rangle_{\H_{j-1}}
=\big\langle\eta,g\hspace{1pt}U_{j,k}\hspace{1pt}\eta\big\rangle_{\H_{j-1}},
$$
with $g\in\linf(\T^{j-1})$. So, one has $U_{j,k}\in C^1(A)$ with
$[A,U_{j,k}]=g\hspace{1pt}U_{j,k}$ due to the density of $C^\infty(\T^{j-1})$ in
$\dom(A)$.

To show that $U_{j,k}\in C^{1+0}(A)$, one has to check that
$
\int_0^1\frac{\d t}t\hspace{1pt}
\big\|\e^{-itA}[A,U_{j,k}]\e^{itA}-[A,U_{j,k}]\big\|_{\B(\H_{j-1})}<\infty
$.
But since $[A,U_{j,k}]=g\hspace{1pt}U_{j,k}$ with $U_{j,k}\in C^1(A)$, one is reduced
to showing that
\begin{align*}
&\int_0^1\frac{\d t}t\hspace{1pt}\big\|\e^{-itA}g\e^{itA}-g\big\|_{\B(\H_{j-1})}
<\infty\\
&\iff\int_0^{(2\pi kb_{j,j-1})^{-1}}\frac{\d s}s\,
\big\|V_{s,j-1}(\partial_{j-1}h_{j-1})V_{-s,j-1}
-(\partial_{j-1}h_{j-1})\big\|_{\B(\H_{j-1})}<\infty,
\end{align*}
which is is readily verified due to the uniform Dini condition of
$\partial_{j-1}h_{j-1}$ in the variable $x_{j-1}$.
\end{proof}

Since $U_{j,k}\in C^1(A)$, we know from Section \ref{Sec_com} that the operator
$$
A_n\eta
:=\frac1n\sum_{\ell=0}^{n-1}
U_{j,k}^{-\ell}\hspace{1pt}A\hspace{1pt}U_{j,k}^\ell\hspace{1pt}\eta
=\frac1n\sum_{\ell=0}^{n-1}U_{j,k}^{-\ell}\big[A,U_{j,k}^\ell\big]\eta
+A\hspace{1pt}\eta,\quad n\in\N^*,~\eta\in\dom(A_n):=\dom(A),
$$
is self-adjoint. In the next lemma, we prove regularity properties of the operator
$U_{j,k}$ with respect to $A_n$ and the strict Mourre estimate for $U_{j,k}$. The
averages of the function $g$ along the flow $\{T_{j-1}^\ell\}_{\ell\in\Z}$, \ie,
$$
g_n:=\frac1n\sum_{\ell=0}^{n-1}g\circ T_{j-1}^{-\ell}\hspace{1pt},\quad n\in\N^*,
$$
appear in a natural way.

\begin{Lemma}[Strict Mourre estimate for $U_{j,k}$]\label{Mourre_Ujk}
Let $j\in\{2,\ldots,d\}$ and $k\in\Z\setminus\{0\}$. Assume that $h_{j-1}$ is of
class $C^1$ in the variable $x_{j-1}$ and that $\partial_{j-1}h_{j-1}$ satisfies a
uniform Dini condition in the variable $x_{j-1}$. Then,
\begin{enumerate}
\item[(a)] one has $U_{j,k}\in C^{1+0}(A_n)$ with
$[A_n,U_{j,k}]=g_n\hspace{1pt}U_{j,k}$\hspace{1pt},
\item[(b)] if $n$ is big enough, one has $g_n>0$ and $(U_{j,k})^*[A_n,U_{j,k}]\ge a$
with $a:=\inf_{x\in\T^{d-1}}g_n(x)>0$.
\end{enumerate}
\end{Lemma}

\begin{proof}
(a) We know from Lemma \ref{UetA} that $U_{j,k}\in C^{1+0}(A)$. So, it follows from
the abstract result \cite[Lemma~4.1]{FRT12} that $U_{j,k}\in C^{1+0}(A_n)$ with
$
[A_n,U_{j,k}]
=\frac1n\sum_{\ell=0}^{n-1}U_{j,k}^{-\ell}\hspace{1pt}[A,U_{j,k}]\hspace{1pt}
U_{j,k}^\ell
$.
Using the equality $[A,U_{j,k}]=g\hspace{1pt}U_{j,k}$, one thus obtains that
\begin{align*}
\big[A_n,U_{j,k}\big]
=\left(\frac1n\sum_{\ell=0}^{n-1}U_{j,k}^{-\ell}\hspace{1pt}g\hspace{1pt}
U_{j,k}^\ell\right)U_{j,k}
=\left(\frac1n\sum_{\ell=0}^{n-1}g\circ T_{j-1}^{-\ell}\right)U_{j,k}
=g_n\hspace{1pt}U_{j,k}\hspace{1pt},
\end{align*}
which concludes the proof of the claim.

(b) It is known from \cite[Thm.~2.1]{Fur61} that the transformation $T_{j-1}$ is
uniquely ergodic. So, one has that
$$
\lim_{n\to\infty}g_n
=\int_{\T^{j-1}}\d\mu_{j-1}\,g
=1+i(b_{j,j-1})^{-1}\big\langle1,P_{j-1}h_{j-1}\big\rangle_{\H_{j-1}}
=1
$$
uniformly on $\T^{j-1}$. Therefore, $g_n>0$ if $n$ is big enough, and point (a)
implies that
$$
(U_{j,k})^*[A_n,U_{j,k}]=(U_{j,k})^*g_n\hspace{1pt}U_{j,k}\ge a,
$$
with $a=\inf_{x\in\T^{d-1}}g_n(x)>0$.
\end{proof}

Using what precedes, we can determine the spectral properties of the operator $W_d$
(see \eqref{def_W_d} for the definition of $W_d$)\hspace{1pt}:

\begin{Theorem}[Spectral properties of $W_d$]\label{Thm_Fur}
For each $j\in\{2,\ldots,d\}$, assume that $h_{j-1}$ is of class $C^1$ in the
variable $x_{j-1}$ and that $\partial_{j-1}h_{j-1}$ satisfies a uniform Dini
condition in the variable $x_{j-1}$. Then, $W_d$ has countable Lebesgue spectrum in
the orthocomplement of $\hspace{1pt}\H_1$.
\end{Theorem}

\begin{proof}
Let $j\in\{2,\ldots,d\}$ and $k\in\Z\setminus\{0\}$. Then, we know from Lemma
\ref{Mourre_Ujk} that $U_{j,k}\in C^{1+0}(A_n)$ and that $U_{j,k}$ satisfies a strict
Mourre estimate on all of $\mathbb S^1$. It follows by Theorem \ref{thm_unitary} that
$U_{j,k}$ has a purely absolutely continuous spectrum in $\H_{j,k}$, and thus that
$W_d$ has purely purely absolutely continuous spectrum in the orthocomplement of
$\H_1$ due to the orthogonal decomposition \eqref{dec_ortho}. Since $W_1$ has pure
point spectrum and $T_d$ is ergodic, it follows from the standard purity law (see
\cite[Thm.~8]{Goo99}) that $W_d$ has countable Lebesgue spectrum in the
orthocomplement of $\H_1$.
\end{proof}

The result of Theorem \ref{Thm_Fur} is not optimal; the nature of the spectrum of
$W_d$ has already been determined by A. Iwanik, M. Lema\'nzyk and D. Rudolph in
\cite[Cor.~3]{ILR93} under a slightly weaker assumption ($\partial_{j-1}h_{j-1}$ of
bounded variation in the variable $x_{j-1}$ instead of Dini-continuous). Even so, our
proof is of independent interest since it is completely new and does not rely on the
study of the Fourier coefficients of the spectral measure.

As a final comment, we note that it would be interesting to see if the technics of
this section could be adapted to variants of Furstenberg transformations (such as
the ones studied in \cite{Hah65}, \cite{JV09} or \cite{Mil86}).

%--------------------------------------------------------------------------------------
\section*{Acknowledgements}
%--------------------------------------------------------------------------------------

The author thanks S. Richard for his idea to take into account the unique ergodicity
in \cite[Sec.~3.3]{FRT12}; this partly motivated the present paper.

%--------------------------------------------------------------------------------------
%\bibliography{../bibliographie/bibliographie}
%--------------------------------------------------------------------------------------

\def\polhk#1{\setbox0=\hbox{#1}{\ooalign{\hidewidth
  \lower1.5ex\hbox{`}\hidewidth\crcr\unhbox0}}}
  \def\polhk#1{\setbox0=\hbox{#1}{\ooalign{\hidewidth
  \lower1.5ex\hbox{`}\hidewidth\crcr\unhbox0}}} \def\cprime{$'$}
  \def\cprime{$'$}

%--------------------------------------------------------------------------------------

\end{document}